\documentclass{amsart}
\usepackage{amssymb,euscript,amsmath, mathrsfs}
\usepackage[dvips]{graphicx}
\usepackage[dvips]{color}
\usepackage{epsfig}

\newcounter{ENUM}

\def\fS{{\mathfrak S}}

\def\ds{\displaystyle}
\newcommand{\bm}[1]{{\boldsymbol{#1}}}
\def\0{\bm{0}}
\def\1{\bm{1}}
\def\ba{\bm{a}}
\def\bb{\bm{b}}

\def\be{\bm{e}}
\def\f{\bm{f}}
\def\bg{\bm{g}}

\def\bi{\bm{i}}
\def\bj{\bm{j}}

\def\bp{\bm{p}}
\def\bq{\bm{q}}

\def\s{\bm{s}}
\def\t{\bm{t}}
\def\v{\bm{v}}
\def\u{\bm{u}}
\def\w{\bm{w}}
\def\x{\bm{x}}
\def\y{\bm{y}}

\def\balpha{\bm{\alpha}}

\def\Br{\operatorname{Br}}

\def\Perm{\operatorname{Perm}}
\def\Ehr{\operatorname{Ehr}}
\def\quot{\operatorname{quot}}

\def\ncone{\operatorname{ncone}}
\def\fcone{\operatorname{fcone}}
\def\lin{\operatorname{lin}}

\def\Z{{\mathbb Z}}
\def\N{{\mathbb N}}

\def\P{{\mathbb P}}
\def\Q{{\mathbb Q}}
\def\R{{\mathbb R}}
\def\C{{\mathbb C}}

\def\G{{\mathbb G}}

\def\cB{{\mathcal B}}
\def\cC{{\mathcal C}}
\def\cF{{\mathcal F}}

\def\cN{{\mathcal N}}
\def\cO{{\mathcal O}}
\def\cP{{\mathcal P}}
\def\cQ{{\mathcal Q}}

\def\cY{{\mathcal Y}}
\def\cZ{{\mathcal Z}}

\def\i{{\mathfrak i}}

\def\Br{\operatorname{Br}}

\def\conv{\mathrm{conv}}
\def\aff{\mathrm{aff}}
\def\lin{\mathrm{lin}}
\def\nvol{\mathrm{nvol}}
\def\vol{\mathrm{Vol}}

\def\relint{\mathrm{relint}}

\def\fcone{\mathrm{fcone}}
\def\tes{\operatorname{Tes}}

\def\Re{\mathfrak Re}

\newtheorem{thm}{Theorem}[subsection]
\newtheorem{thmsec}{Theorem}[section]
\newtheorem{prop}[thm]{Proposition}
\newtheorem{lem}[thm]{Lemma}
\newtheorem{cor}[thm]{Corollary}
\newtheorem{conj}[thm]{Conjecture}

\theoremstyle{definition}
\newtheorem{defn}[thm]{Definition}
\newtheorem{ques}[thm]{Question}
\newtheorem{quessec}[thmsec]{Question}
\newtheorem{ex}[thm]{Example}

\theoremstyle{remark}

\newtheorem{rem}[thm]{Remark}

\numberwithin{equation}{section}

\usepackage{enumerate}
\usepackage{mathtools}
\usepackage{url}
\usepackage{multicol}
\usepackage{pdfpages}
\usepackage{tikz}
\usetikzlibrary{calc}
\usepackage{pgfplots}
\usepackage{subfigure}

\newcommand{\h}{h^*}
\newcommand\commentout[1]{ }
\makeatletter
\renewcommand\subsubsection{\@startsection{subsubsection}{3}%
  \z@{.5\linespacing\@plus.7\linespacing}{-.5em}%
  {\normalfont\bfseries}}\makeatother

\def\multiset#1#2{\ensuremath{\left(\kern-.3em\left(\genfrac{}{}{0pt}{}{#1}{#2}\right)\kern-.3em\right)}}

\def\Poly{\operatorname{Poly}}
\def\PS{ {\mathcal{PS}}}
\def\KP{ {\mathcal{KP}}}
\def\CRY{ {\mathcal{CRY}}}
\def\out{\operatorname{out}}
\def\iin{\operatorname{in}}
\def\codeg{\operatorname{codeg}}

\subjclass[2010]{Primary 52B20; Secondary 05A15, 90C57}
\address{Department of Mathematics, University of California, One Shields Avenue, Davis, California 95616.}
\email{fuliu@math.ucdavis.edu}

\begin{document}
\title{On positivity of Ehrhart polynomials}
\author{Fu Liu}
\begin{abstract}
Ehrhart discovered that the function that counts the number of lattice points in dilations of an integral polytope is a polynomial. We call the coefficients of this polynomial Ehrhart coefficients, and say a polytope is Ehrhart positive if all Ehrhart coefficients are positive (which is not true for all integral polytopes). 
The main purpose of this article is to survey interesting families of polytopes that are known to be Ehrhart positive and discuss the reasons from which their Ehrhart positivity follows. We also include examples of polytopes that have negative Ehrhart coefficients and polytopes that are conjectured to be Ehrhart positive, as well as pose a few relevant questions.  

\end{abstract}

\maketitle

\section{Introduction}

A \emph{polyhedron} in the $D$-dimensional Euclidean space $\R^D$ is the solution set of a finite set of linear inequalities:
\[ P = \left\{ \x \in \R^D \ : \ \sum_{j=1}^D a_{i,j} x_j \le b_i \text{ for $i \in I$} \right\},\]
where $a_{i,j} \in \R$, $b_i \in \R$ and $I$ is a finite set of indices.
A \emph{polytope} is a bounded polyhedron.
Equivalently, a \emph{polytope} in $\R^D$ can also be defined as the convex hull of finitely many points in $\R^D.$ 
We assume readers are familiar with basic definitions such as \emph{faces} and \emph{dimensions} of polytopes as presented in \cite{zie}.
In this paper, the letter $d$ usually denotes the dimension of a polytope and $D$ denotes the dimension of the ambient space. For majority of the examples presented here, we either have $d=D$ or $d=D-1.$ 

A \emph{lattice point} or an \emph{integral point} is a point in $\Z^D.$ Counting lattice points 
inside polytopes is a fundamental and useful step in many mathematical analyses. A lot of combinatorial structures can be counted as lattice points of polytopes. For example, matchings on graphs \cite{lovaszplummer}, t-designs \cite{moura},  (semi-)magic squares \cite[Chapter 6]{beckrobins}, and linear extensions of posets \cite{stanley2poset} are all of this form. Counting lattice points not only appears in combinatorial problems, it also appears,  for instance, in the context of representation theory \cite{kirillov, schmidtbincer}, algebraic geometry \cite{fulton}, statistics \cite{diaconisgangolli, FienbergMMS}, and number theory \cite{beck, nijehuiswilf}.

One approach to study the question of computing the number of lattice points 
in a polytope $P$ is to consider a more general counting problem: For any nonnegative integer $t,$ let $tP := \{ t \x \ : \ \x \in P\}$ be the \emph{$t^{th}$ dilation of $P$}, and then consider the function 
\[
i(P, t) := |tP \cap \Z^D|, 
\]
which counts the number of lattice points in $tP.$
We say two polytopes $P$ and $Q$ are \emph{unimodularly equivalent}\footnote{Unimodular equivalence is sometimes called \emph{integral equivalence}, e.g., in \cite{StaPit}.} if there exists an affine transformation from the affine hull $\aff(P)$ of $P$ to the affine hull $\aff(Q)$ of $Q$ that induces a bijection from lattice points in $\aff(P)$ to lattice points in $\aff(Q)$. Such an affine transformation is called a \emph{unimodular transformation}. It is easy to see that if two polytopes $P$ and $Q$ are unimodularly equivalent, then $i(P,t) = i(Q,t).$

An \emph{integral polytope} (or a \emph{lattice polytope}) is a polytope whose vertices are lattice points.
In the 1960's Eug\`{e}ne Ehrhart \cite{Ehrhart} discovered that the function $i(P,t)$ has nice properties when $P$ is an integral polytope.
\begin{thmsec}[Ehrhart]\label{thm:ehrhart}
For any integral $d$-polytope $P,$ the function $i(P,t)$ is always a polynomial (with real coefficients) of degree $d$ in $t.$
\end{thmsec}
Thus, we call $i(P,t)$ the \emph{Ehrhart polynomial} of an integral polytope $P$, and call the coefficients of $i(P,t)$ the \emph{Ehrhart coefficients} of $P.$ Note that Ehrhart's theorem can be extended to rational polytopes with the concept of a quasi-polynomial; however, we will focus on integral polytopes in this article.

There is much work on the Ehrhart coefficients of integral polytopes. In the 1990's, many people studied the problem of counting lattice points inside 
integral (or more generally rational) polytopes \cite{brion-vergne1,cappell-shaneson, kantor-khovanskii, pommersheim, pukhlikov-khovanskii} by using the theory of toric varieties. Although explicit formulas for coefficients of Ehrhart polynomials can be deduced from these results, they are often quite complicated. Only three coefficients of $i(P,t)$ have simple known forms for arbitrary integral polytopes $P$: the leading coefficient is equal to the normalized volume of $P$, the second coefficient is one half of the sum of the normalized volumes of facets, and the constant term is always $1$. 

Although these three coefficients can be described in terms of volumes (recall $1$ is the normalized volume of a point), and thus are positive, it is not true that all the coefficients of $i(P,t)$ are positive. (The first counterexample comes up in dimension $3$ known as the \emph{Reeve tetrahedron}; see \S \ref{subsec:reeve}.) We say a polytope has \emph{Ehrhart positivity} or is \emph{Ehrhart positive} if it has positive Ehrhart coefficients. It is natural to ask the following question:

\begin{quessec}\label{ques:positive}
Which families of integral polytopes have Ehrhart positivity? 
\end{quessec}

This turns out to be a challenging question. Even though multiple families of polytopes have been shown to be Ehrhart positive in the literature, the techniques involved are (almost) all different. In Section \ref{sec:examples}, we will survey families of polytopes with the Ehrhart positivity property, discussing different reasons why they have this property. 
In particular, as a consequence of the techniques discussed in \S \ref{subsec:higher}, one can show that any combinatorial type of rational polytopes can be realized as an integral polytope that is Ehrhart positive (see Theorem \ref{thm:notcomb}). 
This result indicates that Ehrhart positivity is \emph{not} a combinatorial property. Therefore, it is desirable to find more geometric methods to prove Ehrhart positivity. 
In Section \ref{sec:mcmullen}, we introduce such a tool called \emph{McMullen's formula}, which we use to give a refinement of Ehrhart positivity, called \emph{$\alpha$-positivity}. We then use this tool to attack the Ehrhart positivity conjecture of ``generalized permutohedra'', a family of polytopes introduced by Postnikov \cite{post} and report partial progress on this conjecture. 
In Section \ref{sec:negative}, we include negative results on Question \ref{ques:positive}, presenting examples with negative Ehrhart coefficients. In particular, we will discuss progress on a question asked and studied by Hibi, Higashitani, Tsuchiya and Yoshida \cite{HibHigTsuYos} on all possible sign patterns of Ehrhart coefficients (see \S \ref{subsec:possible}). Note that this question can be considered to be a refinement of Question \ref{ques:positive}. 
Finally, in Section \ref{sec:future}, we include various conjectures on Ehrhart positivity, and pose related questions. 

We finish our introduction with the following remark on the coefficients of the \emph{$h^*$-polynomial}, which is closely related to the Ehrhart polynomials.
\subsection*{Remark on $h^*$-vector}
One method of proving Ehrhart's theorem 
(Theorem \ref{thm:ehrhart}) is by considering the \emph{Ehrhart series} of a $d$-dimensional integral polytope $P$:
\[ \Ehr_P(z) :=\sum_{t \ge 0} i(P,t) z^t.\]
It turns out that Theorem \ref{thm:ehrhart} is equivalent to the existence of a polynomial $h^*_P(z)$ of degree at most $d$ such that $h^*_P(1) \neq 0$ and
\[ \Ehr_P(z) = \frac{h^*_P(z)}{(1-z)^{d+1}}.\]
See \cite[Chapter 4]{stanleyec1ed2} for a statement for the above equivalence result and a proof for Ehrhart's theorem. 

We call $h^*_P(z)$ the \emph{$h^*$-polynomial} of $P$, and the vector $(h^*_{0}, h^*_{1}, \dots, h^*_{d})$, where $h^*_i$ is the coefficient of $z^i$ in $h^*_P(z),$ the \emph{$h^*$-vector} of $P.$
One can recover the Ehrhart polynomial of a $d$-dimensional integral polytope $P$ easily from its $h^*$-vector:
\begin{equation}
i(P,t) = \sum_{j=0}^d h^*_j \binom{t + d -j}{d}.  \label{equ:h2e}
\end{equation}
It is a well-known result due to Stanley that the entries in $h^*$-vectors are all nonnegative integers \cite{stanleydecomp} in contrast to the fact that Ehrhart coefficients could be negative. 
As a consequence, positivity is not such an interesting question for $h^*$-polynomials. Nevertheless, active research have been conducted in other directions.

The most natural question probably is: for each $d$, can we give a complete characterization for all possible $h^*$-vector of $d$-dimensional integral polytopes? For $d=2,$ the answer was first provided in 1976 by Scott \cite{Scott} known as \emph{Scott's condition}. However, for $d \ge 3,$ the question is wide open.  
A lot of work has been done in the literature on searching for inequalities and equalities satisfied by $h^*$-vectors. Most of them were discovered by Hibi \cite{hibi1990, hibi1992, hibi, hibi1994} and Stanley \cite{stanleydecomp, stanley1991} in the 1990s using commutative algebra and combinatorial methods. In 2009, Stapledon \cite{stapledon2009} contributes more inequalities using the idea of degree and codegree of a polytope. 
Known equalities on $h^*$-vectors include 
\begin{equation}\label{equ:knownh}
\sum_{i=0}^d h^*_i = d! \vol(P), \quad h^*_0 = 1, \quad h^*_1 = i(P)-(d+1), \quad h^*_d = |\relint(P) \cap \Z^D|.
\end{equation}
Please see \cite{beck-deloera-et2005, stapledon2009} for lists of known inequalities. Recently, instead of focusing on inequalities satisfied by all polytopes, much work has been done on finding inequalities for polytopes under certain constraints. For example, Treutlein \cite{treutlein} shows that the necessary statement of Scott's condition holds for any integral polytope whose $h^*$-polynomial is of degree at most $2,$ i.e., $h_i^* = 0$ for all $i > 2.$ Most recently, Balletti and Higashitani \cite{BalHig} improve the result further to any integral polytope whose $h^*$-polynomial satisfies $h_3^*=0.$

Another question that comes up a lot in the context of $h^*$-vector is the unimodality question. A sequence of real numbers $c_0, c_1, \dots, c_d$ is \emph{unimodal} if there exists $0 \le j \le d$ such that $c_0 \le c_1 \le \cdots \le c_j \ge c_{j+1} \ge \cdots \ge c_d.$ It is well-known that a nonnegative sequence is unimodal if it has ``no internal zeros'' and is ``log-concave'', and furthermore, log-concavity follows from another property called ``real-rootedness''. Please see surveys by Stanley \cite{stanleylogunimodal} and Brenti \cite{brentilogunimodal} on log-concave and unimodal sequences, and a survey by Br\"and\'en \cite{branden-unilogreal} with a more general discussion on unimodality, log-concavity and real-rootedness. Recently, Braun \cite{braun-survey-unimodal} wrote a survey on unimodality problem of the $h^*$-vector of integral polytopes, discussing a wide range of tools (including but not limited to the techniques mentioned in the aforementioned surveys) to attack this problem. 
Finally, we would like to remark that even though Ehrhart coefficients and $h^*$-vectors are related by \eqref{equ:h2e}, there is no general implication between Ehrhart positivity and $h^*$-unimodality \cite{ehrhart-uni-pos}.

\subsection*{Acknowledgements} The author is partially supported by a grant from the Simons Foundation \#426756. The writing was completed when the author was attending the program ``Geometric and Topological Combinatorics'' at the Mathematical Sciences Research Institute in Berkeley, California, during the Fall 2017 semester, and she was partially supported by the NSF grant DMS-1440140.

The author would like to thank Gabriele Balletti, Ben Braun, Federico Castillo, Ron King, Akihiro Higashitani, Sam Hopkins, Thomas McConville, Karola M\'esz\'aros, Alejandro Morales, Benjamin Nill, Andreas Paffenholz, Alex Postnikov, Richard Stanley, Liam Solus, and Akiyoshi Tsuchiya for valuable discussions and helpful suggestions.
The author is particularly grateful to Federico Castillo and Alejandro Morales for their help in putting together some figures and data used in this article.  

Finally, the author is thankful to the two anonymous referees for their careful reading of this article and various insightful comments and suggestions. 

\section{Polytopes with Ehrhart positivity}\label{sec:examples}

In the literature, there are multiple interesting families of polytopes shown to be Ehrhart positive using very different techniques. In this section, we put together a collection of such families, separating them into four categories based on 
the reasons why they are Ehrhart positive.
However, we make no attempt to give a comprehensive account of \emph{all} families with this property. 
We also note that as the leading coefficient of $i(P,t)$ is the volume of $P,$ one can often extract a formula for volume from descriptions for Ehrhart polynomials we give below. However, we will focus only on results on Ehrhart polynomials here, and omit related formulas for volumes.

In this article, we use bold letters to denote both vectors and points in $\R^D.$ For example, $\be_i$ denotes both the $i$th vector in the standard basis and the point $(0, \dots, 0, 1, 0, \dots, 0)$ where $1$ is in the $i$th position.

For convenience, we use $\N$ to denote the set of nonnegative integers, and $\P$ the set of positive integers. 

\subsection{Products of positive linear polynomials} \label{subsec:explicit}

In this part, we present families of polytopes whose Ehrhart polynomials can be described explicitly, which can be shown to have positive coefficients using the following naive lemma.

\begin{lem}\label{lem:product}
Suppose a polynomial $f(t)$ is either 
\begin{enumerate}[(a)]
	\item \label{item:prod} a product of linear polynomials with positive coefficients, or
	\item \label{item:sum} a sum of products of linear polynomials with positive coefficients. \end{enumerate}
Then $f(t)$ has positive coefficients.
\end{lem}

We start with the two simplest families of polytopes: \emph{unit cubes} and \emph{standard simplices}, whose Ehrhart polynomials fit into situation \eqref{item:prod} of Lemma \ref{lem:product}, and thus Ehrhart positivity follows. As these are the first examples of Ehrhart polynomials in this article, and the computations 
are straightforward, we include all the details. For most of the remaining examples we discuss in this paper, we only state the results without providing detailed proofs.

\subsubsection{Unit cubes} \label{cube}
The $d$-dimensional \emph{unit cube}, denoted by ${\square_d}$, is the convex hull of all points in $\R^d$ with coordinates in $\{0,1\},$ i.e.,
\[ \square_d := \conv \{ \x = (x_1, x_2, \dots, x_d) \in \R^d \ : \ x_i = 0 \text{ or } 1 \text{ for $i=1,2, \dots, d$}\}.\]
It is easy to verify that the unit cube is the solution set to the following linear system of inequalities: 
\[ \square_d = \{ \x =(x_1, x_2, \dots, x_d) \in \R^d \ : \ 0 \le x_i \le 1 \text{ for $i=1,2, \dots, d$}\},\]
Then for any $t \in \N,$
\[ t \square_d = \{ \x =(x_1, x_2, \dots, x_d) \in \R^d \ : \ 0 \le x_i \le t \text{ for $i=1,2, \dots, d$}\}.\]
Thus,
\[  t \square_d \cap \Z^d = \{ \x =(x_1, x_2, \dots, x_d) \in \Z^d \ : \ 0 \le x_i \le t \text{ for $i=1,2, \dots, d$}\}.\]
For each $i,$ the number of integers $x_i$ such that $0 \le x_i \le t$ is $t+1.$ Thus,
\[ i(\square_d, t) = (t+1)^d.\] 

\subsubsection{Standard simplices}\label{simplex}
The $d$-dimensional \emph{standard simplex}, denoted by $\Delta_d,$ is the convex hull of all the elements in the standard basis $\be_1, \be_2, \dots, \be_{d+1}$ of $\R^{d+1}:$
\[ \Delta_d := \conv\{ \be_1, \be_2, \dots, \be_{d+1} \}.\]
One checks that $\Delta_d$ can also be defined by the following linear system:
\[ \sum_{j=1}^{d+1} x_j = 1, \text{ and  } x_i \ge 0 \ \text{ for $i = 1, 2, \dots, d+1$}.\]
Hence, for any $t \in \N,$
\small
\[ t \Delta_d = \left\{ \x =(x_1, x_2, \dots, x_{d+1}) \in \R^{d+1} \ : \ \sum_{j=1}^{d+1} x_j = t, \text{ and }  x_i \ge 0 \text{ for $i=1, 2, \dots, d+1$} \right\},\]
\normalsize and
\small \[ t \Delta_d \cap \Z^{d+1} = \left\{ \x =(x_1, x_2, \dots, x_{d+1}) \in \Z^{d+1} \ : \ \sum_{j=1}^{d+1} x_j = t, \text{ and }  x_i \ge 0 \text{ for $i = 1, 2, \dots, d+1$} \right\}.\]
\normalsize
Hence, $i(\Delta_d,t)$ counts the number of nonnegative integer solutions to 
\[ x_1 + x_2 + \cdots + x_{d+1} = t.\]
This is a classic combinatorial problem which is the same as counting the number of weak compositions of $t$ into $d+1$ parts (see \cite[Page 18]{stanleyec1ed2}), and the answer is given by 
\[ i(\Delta_d, t) = \binom{t+d}{d}.\]

As we mentioned above, Ehrhart positivity of unit cubes and standard simplices follows from situation \eqref{item:prod} of Lemma \ref{lem:product}. Next, we present two families of examples with Ehrhart polynomials in the form of situation \eqref{item:sum} of Lemma \ref{lem:product}.

\subsubsection{Pitman-Stanley polytopes} \label{subsubsec:PitSta}

Let $\ba = (a_1, \dots, a_{d}) \in \N^{d}$. The following polytope is introduced and studied by Pitman and Stanley \cite{StaPit}:
\[ \PS_d(\ba) := \left\{ \x \in \R^d \ : \ x_i \ge 0 \text{ and } \sum_{j=1}^i x_j \le \sum_{j=1}^i a_i, \ \text{for $i=1,2,\dots, d$}\right\},\]
hence we call it a \emph{Pitman-Stanley polytope}.



Pitman and Stanley gave an explicit formula \cite[Formula (33)]{StaPit} for computing the number of lattice points in $\PS_d(\ba)$, from which a formula for the Ehrhart polynomial of $\PS_d(\ba)$ immediately follows. Recall $\multiset{x}{y} = \binom{x+y-1}{y}.$

\begin{thm}[Pitman-Stanley]\label{thm:pitsta}
Let 
\small
\[ I_d := \left\{ \bi = (i_1, i_2, \dots, i_d) \in \N^d \ : \ \sum_{j=1}^d i_j = d, \text{ and } \sum_{j=1}^k i_j \ge k  \text{ for } k=1,2,\dots d-1\right\}.\]
\normalsize
Then the Ehrhart polynomial of $\PS_d(\ba)$ is given by
\begin{equation}\label{equ:pitsta}
i(\PS_d(\ba), t) = \sum_{\bi \in I_d} \multiset{a_1t+1}{i_1} \prod_{k=2}^d \multiset{a_k t}{i_k}.
\end{equation}
\end{thm}

For each $\i$, both $\multiset{a_1t+1}{i_1}$ and $\multiset{a_k t}{i_k}$ are products of linear polynomials in $t$ with positive coefficients, 
so it follows from Lemma \ref{lem:product}/\eqref{item:sum} that any Pitman-Stanley polytope $\PS_d(\ba)$ is Ehrhart positive. 

Pitman-Stanley polytopes are contained in two different more general families of polytopes: flow polytopes and generalized permutohedra. For each of these two bigger families of polytopes, formulas for Ehrhart polynomials of some subfamily have been derived, generalizing Formula \eqref{equ:pitsta}. 
We present results on flow polytopes in the next part below, while the results on generalized permutohedra are postponed to \S \ref{subsubsec:typey} as part of a general discussion on the Ehrhart positivity conjecture of generalized permutohedra in Section \ref{sec:mcmullen}.

\subsubsection{Subfamilies of flow polytopes} \label{subsubsec:flow}
Let $G$ be a (loopless) directed acyclic connected graph on $[n+1]=\{1,2,\dots,n+1\}$ such that each edge $\{i,j\}$ with $i < j$ is always directed from $i$ to $j$. Hence, we denote the edge by $(i,j)$ to indicate the orientation. 
For any $\ba = (a_1, a_2, \dots, a_n) \in \N^n,$ we associate to it another vector 
\begin{equation}
\bar{\ba} := \left(a_1,\dots, a_n, -\sum_{i=1}^n a_i\right).
\label{equ:bara}
\end{equation}
An \emph{$\bar{\ba}$-flow} on $G$ is a vector $\f = (f(e))_{e \in E(G)} \in \left(\R_{\ge 0}\right)^{E(G)}$ such that for $i=1,2,\dots, n,$ we have
\[ \sum_{e=(g,i) \in E(G)} f(e) + a_i = \sum_{e=(i,j) \in E(G)} f(e),\]
that is, the \emph{netflow} at vertex $i$ is $a_i.$ Note these conditions imply that the netflow at vertex $n+1$ is $-\sum_{i=1}^n a_i.$
The \emph{flow polytope} $\cF_G(\bar{\ba})$ associated to $G$ and the integer netflow $\bar{\ba}$ as the set of all $\bar{\ba}$-flows $\f$ on $G.$ 

\begin{ex} \label{ex:graphpitsta} Let $G_d^{\PS}$ be the graph on $[d+1]$ with edge set 
	\[\{ (i,i+1), (i,d+1) \ : \ i=1,2,\dots, d\}.\]
	Baldoni and Vergne \cite[Example 16]{BalVer} show that $\cF_{\G_d^{\PS}}(\bar{\ba})$ is unimodularly equivalent to the Pitman-Stanley polytope $\PS_d(\ba)$. 
\end{ex}

For each edge $e=(i,j)$ of $G,$ we associate to it the positive type $A_n$ root $\alpha(e) = \alpha(i,j) = \be_i-\be_j.$ For any $\bb \in \Z^{n+1},$ the \emph{Kostant partition function} $\KP_G$ evaluated at $\bb$ is 
\[ \KP_G(\bb) := \# \left\{ \f = (f(e))_{e \in E(G)} \in \N^{E(G)} \ : \ \sum_{e \in E(G)} f(e) \alpha(e) = \bb \right\}. \]
It is straightforward to verify that for $\ba \in \N^n,$ 
\[ \KP_G(\bar{\ba}) = | \cF_G(\bar{\ba}) \cap \Z^{E(G)}|,\]
i.e., $\KP_G(\bar{\ba})$ counts the number of lattice points in the flow polytope $\cF_G(\bar{\ba})$.
In the literature, various groups of people \cite{lidskii, PosStaFlow2, BalVer, MesMor} obtained formulas for Kostant partition functions, or equivalently, the number of lattice points in flow polytopes. As a consequence, we can easily obtain formulas for the Ehrhart polynomial of $\cF_G(\bar{\ba}).$


\begin{thm}[Lidskii, Postnikov-Stanley, Baldoni-Vergne, M\'esz\'aros-Morales]\label{thm:flow}
	Suppose $G$ is a connected graph on the vertex set $[n + 1]$, with $m$ edges directed $i \to j$ if $i < j$, and with at least one outgoing edge at vertex $i$ for $i = 1, \dots, n.$ 
	Let $\out_k$ (and $\iin_k$, respectively) denote the outdegree (and the indegree, respectively) of vertex $k$ in $G$ minus $1.$

	Then for any $\ba = (a_1, \dots, a_n) \in \N^n,$ the Ehrhart polynomial of $\cF_G(\bar{\ba})$ is given by
	\begin{align}
		& i(\cF_G(\bar{\ba}), t) \nonumber \\
		=& \sum_{\bj} \prod_{k=1}^n \binom{a_k t + \out_k}{j_k}  \cdot \KP_G(j_1-\out_1, j_2-\out_2, \dots, j_n-\out_n, 0), \label{equ:flow0} \\
		=& \sum_{\bj} \prod_{k=1}^n \multiset{a_k t - \iin_k}{j_k}  \cdot \KP_G(j_1-\out_1, j_2-\out_2, \dots, j_n-\out_n, 0), \label{equ:flow}
	\end{align}
where each summation is over all weak compositions $\bj=(j_1, \dots, j_n)$ of $m-n$ that are $\ge (\out_1, \dots, \out_n)$ in dominance order.
\end{thm}

We remark that Lidskii \cite{lidskii} gives a formula for computing Kostant partition functions associated to the complete graph $K_{n+1},$ which yields Formula \eqref{equ:flow0} above with $G = K_{n+1}.$ Postnikov and Stanley \cite[unpublished]{PosStaFlow2} were the first to discover Formula \eqref{equ:flow0} for arbitrary graphs $G$ using the Elliott-MacMahon algorithm. Baldoni and Vergne \cite{BalVer} give a proof for both formulas in Theorem \ref{thm:flow} using residue computation. 
Most recently, M\'esz\'aros and Morales \cite{MesMor} recover Baldoni-Vergne's result by extending ideas of Postnikov and Stanley on the Elliott-MacMahon algorithm and polytopal subdivisions of flow polytopes. 

Formula \eqref{equ:flow} is useful in obtaining positivity results since
\[\multiset{a_kt-\iin_k}{j_k} = \binom{a_k t - \iin_k+j_k-1}{j_k}\] is a product of linear polynomials in $t$ with positive coefficients as long as $\iin_k = 0$ or $-1.$ 
Also, note that $\KP_G(j_1-\out_1, j_2-\out_2, \dots, j_n-\out_n, 0)$ is nonnegative and $i(\cF_G(\bar{\ba}),t) \neq 0.$ The following result immediately follows from Lemma \ref{lem:product}/\eqref{item:sum}.

\begin{cor}\label{cor:flow}
  Assume the hypotheses of Theorem \ref{thm:flow}. Assume further that for each vertex $i \in [n]=\{1,2,\dots,n\},$ the indegree of $i$ is either $0$ or $1.$ Then the flow polytope $\cF_G(\bar{\ba})$ is Ehrhart positive.
\end{cor}

We remark that the graph $G^{\PS}_d$ defined in Example \ref{ex:graphpitsta} satisfies the hypothesis of the above corollary. Hence, Ehrhart positivity of the Pitman-Stanley polytope is a special case of Corollary \ref{cor:flow}.




\subsection{Roots with negative real parts}
In this part, we show examples with Ehrhart positivity using the following lemma. 
We use $\Re(z)$ to denote the real part of a complex number $z.$ 
\begin{lem}\label{lem:neg}
Let $p(t)$ be a polynomial in $t$ with real coefficients. If the real part $\Re(r)$ is negative for every root $r$ of $p(t)$, then all the coefficients of $p(t)$ are positive.
\end{lem}
\begin{proof} Let $a >0.$
If $-a < 0$ is a real root of $p(t),$ then $t+a$ is a factor of $p(t).$ If $-a + bi$ is a complex root of $p(t)$ for some $b \in \R,$ then $-a-bi$ must be a root of $p(t)$ as well, which implies that 
\[ (t+a-bi)(t+a+bi) = (t^2 + 2a t + a^2 + b^2) \]
is a factor of $p(t).$ Hence, $p(t)$ is a product of factors with positive coefficients. Thus, our conclusion follows.
\end{proof}

We say that a polynomial (with real coefficients) is \emph{negative-real-part-rooted} or \emph{NRPR} if all of its roots have negative real parts. 
The above lemma implies that if $i(P,t)$ is NRPR, then $P$ is Ehrhart positive. Ehrhart polynomials of unit cubes and standard simplices are all trivially NRPR, as they factor into linear polynomials with positive real coefficients. Hence, we would like to rule them out, and are only interested in examples of Ehrhart polynomials that are nontrivially NRPR. 

It turns out that the following theorem which establishes a connection between roots of the $h^*$-polynomial and roots of the Ehrhart polynomial of a polytope is very useful. 
\begin{thm}[\cite{rodriguez}, Theorem 3.2 of \cite{stanleycycleperm}]\label{thm:root} 
	Let $P$ be a $d$-dimensional integral polytope, let $k$ be the degree of the polynomial $h^*_P(z)$ (so that $0 \le k \le d$), and suppose that every root of $h^*_P(z)$ lies on the circle $\{ z \in \C \ : \ |z|=1\}$ in the complex plane. Then there exists a polynomial $f(t)$ of degree $k$ such that
\[ i(P,t) = f(t) \cdot \prod_{i=1}^{d-k} (t+i),\]
and every root of $f(t)$ has real part $-(1+(d-k))/2.$
\end{thm}

We say a polytope $P$ is \emph{$h^*$-unit-circle-rooted} if the $h^*$-polynomial $h^*_P(z)$ of $P$ has all of its roots on the unit circle of the complex plane. 
Below we introduce three families of polytopes, and show that each polytope $P$ in these families is $h^*$-unit-circle-rooted. 
Therefore, Ehrhart positivity for these families follows from Theorem \ref{thm:root} and Lemma \ref{lem:neg}. 

\subsubsection{Cross-polytopes} \label{subsubsec:cross}
The $d$-dimensional cross-polytope, denoted by $\Diamond_d,$ is a polytope in $\R^d$ defined by
\[ \Diamond_d := \conv\{ \pm \be_1, \pm \be_2, \dots, \pm \be_d \},\] 
or equivalently by the following linear system:
\[ \pm x_1 \pm x_2 \pm \cdots \pm x_d \le 1.\]
Hence, $i(\Diamond_d, t)$ counts the number of integer solutions to
\[ |x_1| + |x_2| + \cdots  |x_d| \le t.\]
Counting the number of integer solutions with exactly $k$ nonzero $x_i$'s for $k=0,1,2,\dots, d,$ we obtain that
\[ i(\Diamond_d, t) = \sum_{k=0}^d 2^k\binom{d}{k}\binom{t}{k}.\]
Unfortunately, it is \emph{not} clear whether the above expression expands positively in powers of $t$.
We compute its Ehrhart series instead. First, notice that $i(\Diamond_d, t)$ counts the number of integer solutions to
\[ |x_1| + |x_2| + \cdots  |x_d| + y = t.\] Hence, 
\[ i(\Diamond_d, t) = \sum f(a_1) f(a_2) \cdots f(a_d)f(b),\]
where the summation is over all weak compositions $(a_1, \dots, a_d, b)$ of $t$ into $d+1$ parts, $g(b)=1$ for all $b \ge 0$ and $f(a) = 1$ if $a=0$ and $f(a)=2$ if $a>0.$ Therefore,
\[ \sum_{t \ge 0} i(\Diamond_d, t) z^t = \prod_{i=1}^d \left( \sum_{a_i \ge 0} f(a_i) z^{a_i} \right) \cdot \sum_{b \ge 0} z^b = \left( \frac{1+z}{1-z} \right)^d \cdot \frac{1}{1-z} = \frac{(1+z)^d}{(1-z)^{d+1}}.\]
Thus, $(1+z)^d$ is the $h^*$-polynomial of the cross-polytope $\Diamond_d.$ Hence, $\Diamond_d$ is $h^*$-unit-circle-rooted, and thus are Ehrhart positive.

\subsubsection{Certain families of $\Delta_{(1,\bq)}$} \label{subsubsec:delta1q} Let $\bq = (q_1, q_2,\dots, q_d) \in \P^d$ be a sequence of positive integers. For each such a vector $\bq$, we define a simplex
\[ \Delta_{(1, \bq)} := \conv \left\{ \be_1, \be_2, \dots, \be_d, - \sum_{i=1}^d q_i \be_i \right\}.\]
In \cite{conrads}, Conrads studied simplices of this form and showed that $\Delta_{(1, \bq)}$ is reflexive if and only if
\[ q_i \text{ divides } 1 + \sum_{j =1}^d q_j, \quad \text{ for $i=1,2,\dots, d$.}\]
Recently, Braun, Davis and Solus studied $\Delta_{(1,\bq)}$ in their investigation of a Conjecture by Hibi and Ohsugi, and they provided a number-theoretic characterization of the $h^*$-polynomial of $\Delta_{(1,\bq)}$ \cite[Theorem 2.5]{BraDavSol}.
\begin{thm}[Braun-Davis-Solus]\label{thm:comph}
	The $h^*$-polynomial of $\Delta_{(1,\bq)}$ is
\begin{equation}\label{equ:h}
h^*\left( \Delta_{(1,\bq)}, z \right) = \sum_{b=0}^{q_1+q_2+\cdots+q_d} z^{\omega(b)},
\end{equation}
	where
	\[ \omega(b) := b - \sum_{i=1}^d \left\lfloor \frac{q_i b}{1+q_1+\cdots + q_d} \right\rfloor.\]
\end{thm}

Formula \eqref{equ:h} allows us to compute the $h^*$-polynomial for $\Delta_{(1,\bq)}$ with special choices of $\bq$ easily. We give two examples below 
such that $\Delta_{(1,\bq)}$ satisfies the hypothesis of Theorem \ref{thm:root} with $k=d.$ 
\begin{ex}[Standard reflexive simplices] \label{ex:refsimplex}
	If we choose $\bq=\1 =(1, 1, \dots, 1) \in \P^d,$ then $\Delta_{(1,\bq)}$ is the $d$-dimensional \emph{standard reflexive simplex}. Note that in this case, we have that $q_1 + q_2 + \cdots + q_d = d.$
	Furthermore, for each $b \in \{0, 1, 2, \dots, d\},$ one can verify that $w(b) = b.$ Hence, 
	\[ h^*\left( \Delta_{(1,\1)}, z \right) = \sum_{b=0}^{d} z^{b} = 1+ z + z^2 +\cdots + z^d.\]
\end{ex}

\begin{ex}[Payne's construction]\label{ex:payne}
	In \cite{payne}, Payne constructed reflexive simplices that do not have unimodal $h^*$-vectors. His construction is equivalent to the simplices $\Delta_{(1, \bq)}$ with the following choices of $\bq:$ Given $r \ge 0, s \ge 3$ and $k \ge r+2,$ let $d = r+sk$ and
	\[ \bq = (q_1, q_2, \dots, q_d) = (\underbrace{1, 1, \dots, 1}_{sk-1 \text{ times}}, \underbrace{s, s, \dots, s}_{r+1 \text{ times}}).\]
	Applying Theorem \ref{thm:comph}, one can obtain 
	\[ h^*\left( \Delta_{(1,\bq)}, z \right) = (1 + z^k + z^{2k} +\cdots + z^{(s-1)k}) ( 1+ z+ z^2 + \cdots + z^{k+r}).\]
\end{ex}


Both Example \ref{ex:refsimplex} and Example \ref{ex:payne} are $h^*$-unit-circle-rooted. Hence, they are Ehrhart positive. We remark that the Ehrhart positivity of $\Delta_{(1,\bq)}$ considered in Example \ref{ex:payne} was first proved by the author and Solus \cite[Theorem 3.2]{ehrhart-uni-pos}.

\subsubsection{One family of order polytopes} \label{subsubsec:order}
Given a finite poset (partially ordered set) $\cP$, the \emph{order polytope}, denoted by $\cO(\cP)$, is the collection of functions $\x \in \R^\cP$ satisfying
\begin{itemize}
	\item $0 \le x_a \le 1,$ for all $a \in \cP,$ and
	\item $x_a \le x_b$, if $a \le b$ in $\cP.$
\end{itemize}
The order polytope $\cO(\cP)$ was first defined and studied by Stanley \cite{stanley2poset}. 
Here we consider a family of order polytopes constructed from a certain family of posets.

Let $\cP_k$ be the ordinal sum of $k$ copies of $2$ element antichains, equivalently,
$\cP_k$ is the poset on the $2k$-element set $\{ a_{i,j} \ : \ i=1,2 \text{ and } j=1,2,\dots, k\}$ satisfying
\[ a_{i,j} \le a_{i',j'} \text{ if and only if } j < j' \text{ or } (i,j) =(i',j').\]
For any $t \in \N,$ the $t^{th}$ dilation $t \cO(\cP_k)$ of $\cO(\cP_k)$ is the collection of $\x = (x_{i,j} \ : \ i=1,2 \text{ and } j=1,2,\dots,k) \in \R^{2k}$ satisfying
\[ 0 \le x_{i,j} \le t, \quad \text{ and } \quad x_{i,j} \le x_{i',j'} \text{ if } j < j'.\]
Hence, $i(\cO(\cP_k),t)$ counts the number of integer solutions $\x$ satisfying the above two conditions. Note that each solution gives a weak composition $(y_1, z_1, y_2, \dots, y_k, z_k, y_{k+1})$ of $t$ into $2k+1$ parts, where 
\begin{align*}
	y_j =& \min(x_{1,j}, x_{2,j}) - \max(x_{1,j-1}, x_{2,j-1}), \quad \text{ for } j=1,2,\dots, k+1, \\
	z_j =& \max(x_{1,j},x_{2,j}) - \min(x_{1,j}, x_{2,j}) = |x_{1,j}-x_{2,j}|, \quad \text{ for } j=1,2,\dots, k,
\end{align*}
and by convention let $\max(x_{1,0},x_{2,0})=0$ and $\min(x_{1,k+1}, x_{2,k+1})=t.$
	Thus,
	\[ i(\cO(\cP_k),t) = \sum g(y_1)f(z_1)g(y_2)f(z_2) \cdots f(z_k) g(y_{k+1}),\]
where the summation is over all weak compositions of $t$ into $2k+1$ parts, $g(y)=1$ for all $y \ge 0,$ and $f(z) = 1$ if $z=0$ and $f(z)=2$ if $z>0.$ Therefore, similar to the calculation for cross-polytopes, we obtain
\[ \sum_{t \ge 0} i(\cO(\cP_k), t) z^t =  \frac{(1+z)^k}{(1-z)^{2k+1}}.\]
Thus, the $h^*$-polynomial of $\cO(\cP_k)$ is $(1+z)^k.$ 
By Lemma \ref{lem:neg} and Theorem \ref{thm:root}, the order polytope $\cO(\cP_k)$ is Ehrhart positive.

\begin{rem}\label{rem:chain}
	Stanley also defined a ``chain order'' polytope $\cC(\cP)$ for each poset $\cP$ \cite[Definition 2.1]{stanley2poset}, and showed that $\cC(\cP)$ is unimodularly equivalent to $\cO(\cP)$ \cite[Theorem 3.2/(b)]{stanley2poset}, from which it follows that $i(\cC(\cP),t) = i(\cO(\cP),t).$

Therefore, the conclusions we draw above for the order polytope $\cO(\cP_k)$ all hold for the chain polytope $\cC(\cP_k).$
\end{rem}

It turns out that the polytopes studied in \S \ref{subsubsec:cross} and \S \ref{subsubsec:delta1q} are ``reflexive'' polytopes, and the order polytopes studied in \S \ref{subsubsec:order} are ``Gorenstein'' polytopes. These are not coincidences as we will discuss below.

\subsection*{Connection to reflexivity and Gorensteinness} An integral polytope $P$ in $\R^D$ is \emph{reflexive} (up to lattice translation) if the origin is in the interior of $P$ and its \emph{dual}
\[ P^\vee:= \{ \y \in (\R^D)^* \ : \ \langle \y, \x \rangle \le 1 \ \forall \x \in P\} \]
is also an integral polytope, where $(\R^D)^*$ is the dual space of $\R^D.$

It follows from the Macdonald Reciprocity Theorem \cite{macdonald} that if an integral polytope $P$ is reflexive, then the roots of $i(P,t)$ are symmetrically distributed with respect to the line $\{ z \in \C \ : \ \Re(z) = -1/2\}$ in the complex plane. Bey, Henk and Wills show that the converse is true if 
we include polytopes that are unimodularly equivalent to reflexive polytopes \cite[Proposition 1.8]{BeyHenWil}. 
Recently, Heged\"us, Higashitani and Kasprzyk, in their study of roots of Ehrhart polynomials of reflexive polytopes, give the following result 
\cite[Lemma 1.2]{HegHigKas}.
\begin{lem}[Heged\"us-Higashitani-Kasprzyk]\label{lem:reflexiveroot}
	A $d$-dimensional integral polytope $P$ is reflexive (up to unimodular transformation) if and only if the summation of the $d$ roots of $i(P,t)$ equals to $\displaystyle -d/2.$
\end{lem}

Reflexive polytopes are special cases of a more general family of polytopes: Gorenstein polytopes.
Recall that the \emph{codegree} of $P$ is defined to be 
\[ \codeg(P) := \dim(P) + 1 - \deg\left( h^*_P(z) \right).\]
It is (again) a consequence of the Macdonald Reciprocity Theorem \cite{macdonald} that $\codeg(P)$ is the smallest positive integer $s$ such that $sP$ contains a lattice point in its interior (see, for example, \cite{hibi1992book}). 
A \emph{Gorenstein polytope} is an integral polytope $P$ of codegree $s$ such that $sP$ is a reflexive polytope. 
The work \cite{stanleyhilb} of Stanley gives a nice characterization for the $h^*$-polynomials of Gorenstein polytopes: a $d$-dimensional integral polytope $P$ is a Gorenstein polytope
if and only if its $h^*$-polynomial is \emph{symmetric}, that is, if $h_P^*(z) = \sum_{i=0}^{k} h_i^* z^i$ with $h_k^* \neq 0,$ then $h_i^* = h_{k-i}^*$ for $i=0,1,2,\dots, k.$ Using this, one can easily see that all the examples discussed in \S \ref{subsubsec:cross}, \S \ref{subsubsec:delta1q}, and \S \ref{subsubsec:order} are Gorenstein polytopes.

We now restate Lemma \ref{lem:reflexiveroot} in terms of Gorenstein polytopes.
\begin{lem}\label{lem:gorensteinroot}
	A $d$-dimensional integral polytope $P$ is Gorenstein (up to unimodular transformation) if and only if the summation of the $d$ roots of $i(P,t)$ equals to $\displaystyle -{sd}/{2}$ for some positive integer $s.$

	Furthermore, if the above condition holds, the integer $s$ is the codegree of $P.$
\end{lem}

\begin{proof}
The conclusion of the lemma follows from the observation that a number $t_0$ is a root of $i(P,t)$ if and only if $t_0/s$ is a root of $i(sP,t).$
\end{proof}

\begin{cor}\label{cor:gorenstein}
	Suppose $P$ is a $d$-dimensional polytope that is $h^*$-unit-circle-rooted. Then $P$ is a Gorenstein polytope (up to unimodular transformation). 
	
	Moreover, if the degree of the $h^*$-polynomial $h^*_P(z)$ is $d,$ then $P$ is reflexive.
\end{cor}

\begin{proof}
	Let $k$ be the degree of $h^*_P(z).$ By Theorem \ref{thm:root}, among all the roots of $i(P,t)$, $k$ of them have real parts $-(1+(d-k))/2,$ and the other $(d-k)$ roots are $-1, -2, \dots, -(d-k).$ 
		As $i(P,t)$ is a polynomial with real coefficients, the sum of roots of $i(P,t)$ is the sum of the real parts of roots of $i(P,t)$, which is 
		\[ -(1+(d-k))/2 \cdot k + \sum_{i=1}^{d-k} (-i) = -\frac{1}{2}d(d-k+1).\] 
		Then the conclusions follow from Lemma \ref{lem:gorensteinroot}.
\end{proof}

Therefore, 
when an integral polytope $P$ is $h^*$-unit-circle-rooted, we not only get Ehrhart positivity for $P$ but can also conclude that $P$ is a Gorenstein polytope of codegree $d-k+1$, where $k$ is the degree of the $h^*$-polynomial of $P.$ 

\subsection{Coefficients with combinatorial meanings}


\subsubsection{Zonotopes} \label{subsubsec:zonotope}
In this part, we introduce a special family of polytopes, \emph{zonotopes}, whose Ehrhart coefficients can be described combinatorially. As a consequence, Ehrhart coefficients of a zonotope are not only positive but also positive integers.

The \emph{Minkowski sum} of two polytopes (or sets) $P$ and $Q$ is 
\[ P + Q := \{ \x + \y \ : \ \x \in P, \ \y \in Q.\}.\]
Let $\v_1, \v_2, \dots, \v_n \in \Z^D$ be a set of integer vectors. The \emph{zonotope} $\cZ(\v_1,\v_2,\dots, \v_n)$ associated with this set of vectors is the Minkowski sum of intervals $[0, \v_i],$ where $[0,\v_i]$ is the line segment from the origin to $\v_i.$ Hence,
\[
  {\cZ}(\v_1,\cdots,\v_n) := \sum_{i=1}^n [0, \v_i] = \left\{ \sum_{i=1}^n c_i \v_i \ : \ 0 \le c_i \le 1 \ \text{for $i=1,2,\dots,n$} \right\}.
\]
In \cite[Theorem 2.2]{stanleyzonotope}, Stanley gives a combinatorial description for the Ehrhart coefficients of zonotopes.
\begin{thm}[Stanley] \label{thm:zonotope}
	The coefficient of $t^k$ in $i(\cZ(\v_1,\cdots,\v_n),t)$ is equal to 
	\[ {\sum_{X} h(X)},\]
	where $X$ ranges over all linearly independent $k$-subsets of $\{\v_1,\dots, \v_n\},$ and $h(X)$ is the greatest common divisor of all $k \times k$ minors of the matrix whose column vectors are elements of $X.$
\end{thm}

The main ingredient for the proof of the above theorem is that $\cZ(\v_1, \dots, \v_n)$ can be written as a disjoint union of half open parallelepiped $C_X$ ranging over all linearly independent subsets $X=\{\v_{j_1},\dots, \v_{j_k}\}$ of $\{\v_1,\dots, \v_n\},$ where $C_X$ is generated by $\epsilon_1 \v_{j_1}, \dots, \epsilon_i \v_{j_k}$ for certain choices of $\epsilon_1, \dots, \epsilon_k \in \{-1, 1\}.$ (See \cite[Lemma 2.1]{stanleyzonotope}.) The theorem then follows from the fact that the number of lattice points in the half open parallelepiped $C_X$ is the volume of $C_X$, which can be calculated by $h(X).$

The simplest examples of zonotopes are unit cubes considered in \S \ref{cube}. We may recover the Ehrhart polynomial of a unit cube using Theorem \ref{thm:zonotope}. However, a more interesting example is the \emph{regular permutohedron}.

\begin{ex} \label{ex:regular} 
	The \emph{regular permutohedron}, denoted by $\Pi_d$, is the convex hull of all permutations in $\fS_{d+1};$ that is, 
\[  \Pi_{d} 
:= \conv\{(\sigma(1),\sigma(2),\cdots,\sigma(d+1))\in \R^{d+1}: \sigma\in \fS_{d+1}\}. \]
It is straightforward to check that $\Pi_d$ is a translation of the zonotope
\[ \sum_{1 \le i < j \le d+1} [0, \be_j - \be_i].\]
For any subset $X$ of $\Phi_d := \{ \be_j - \be_i \ : \ 1 \le i < j \le d+1\},$ we let $G_X$ be the graph on vertex set $[d+1]$ and $\{i,j\}$ (with $i < j$) is an edge if and only if $\be_j -\be_i \in X.$ Then it follows from matroid theory that $X$ is linearly independent if and only if $G_X$ is a forest on $[d+1]$. (Recall that a \emph{forest} is a collection of trees, or equivalently, is acyclic.) Furthermore, if $X$ is linearly independent, then $G_X$ is a forest of $d+1-|X|$ trees, and $h(X)$ (described in Theorem \ref{thm:zonotope}) is $1.$

Therefore, we conclude that the coefficient of $t^k$ in $i(\Pi_d, t)$ counts the number of forests on $[d+1]$ that contain exactly $d+1-k$ trees. 
Therefore, we can compute, for example,
\[ i(\Pi_3, t) = 16 t^3 + 15 t^2 + 6t + 1.\]
\end{ex}

\subsubsection{Positivity of a generalized Ehrhart polynomial} The polynomial we discuss in this part is not exactly an Ehrhart polynomial. However, it is closely related to the result on zonotopes we have presented in the last part, and thus is included here.
Galashin, Hopkins, McConville and Postnikov, in their study of root system chip firing \cite{GalHopMccPos2017a}, considered the following lattice points counting problem: Given an integral polytope $P$ and a set of integer vectors $\v_1, \v_2, \dots, \v_n,$ describe the number of lattice point in 
\[ P + \v_1 + \v_2 + \cdots + \v_n = P + \cZ(\v_1, \dots, \v_n). \]
Extending Stanley's idea of decomposing zonotopes into half open parallelepiped, they show \cite[Proof of Theorem 16.1]{GalHopMccPos2017a} that $P + \cZ(\v_1, \dots, \v_n)$ can be written as disjoint union of sets in the form of $F + C_X$ where $F$ is an open face of $P$ and $C_X$ is a half open parallelepiped determined by a linearly independent set $X$ of $\{\v_1,\dots, \v_n\}$. 
\begin{thm}[Galashin-Hopkins-McConville-Postnikov]\label{thm:genzonotope}
Suppose $P$ is an integral polytope in $\R^D$ and $\v_1, \dots, \v_n \in \Z^D$ is a set of integer vectors. Let $\cZ = \cZ(\v_1, \dots, \v_n)$. 
For any $\t = (t_1, \dots, t_n) \in \N_n$, we define $\t \cZ = \cZ(t_1 \v_1, \dots, t_n \v_n).$ Then there exists a polynomial $L(\t) = L(t_1, \dots, t_n)$ in $n$ variables with nonnegative integer coefficients such that
$|(P + \t \cZ) \cap \Z^D| = L(\t).$

In particular, if we take $\t = (t, t, \dots, t),$ then 
\[ |(P + \t \cZ) \cap \Z^D| = |(P+ t \cZ) \cap \Z^D| = L(t, t, \dots, t) \] 
is a polynomial in $t$ of degree $\dim(\cZ)$ with positive integer coefficients.
\end{thm}
Note that the second part of the above theorem was not explicitly stated in \cite[Theorem 16.1]{GalHopMccPos2017a}; but it was a consequence of the techniques used in its proof.

One sees that if we choose $P$ to be the origin, then the above theorem recovers the Ehrhart positivity of zonotopes. However, in contrast with Stanley's results, no explicit formulas were given in \cite{GalHopMccPos2017a} for the positive/nonnegative integer coefficients asserted in Theorem \ref{thm:genzonotope}. 
Recently, Hopkins and Postnikov \cite{HopPos} analyzed techniques used in \cite{GalHopMccPos2017a} further, and provided the desired explicit formula, completing the generalization of Theorem \ref{thm:zonotope}. 

\begin{thm}[Hopkins-Postnikov]
	The homogeneous degree $k$ part of the polynomial $L(\t)$ assumed by Theorem \ref{thm:genzonotope} is given by
	\[\sum_X |\quot_X(P) \cap \quot_X(\Z^D)| \cdot h(X) \cdot \prod_{\v_i \in X} t_i,\]
	where $X$ ranges over all linearly independent $k$-subsets of $\{\v_1,\dots, \v_n\},$ $\quot_X: \R^D \to \R^D/\rm{span}_\R(X)$ is the canonical quotient map, and $h(X)$ is the greatest common divisor of all $k \times k$ minors of the matrix whose column vectors are elements of $X.$

\end{thm}


\subsection{Higher integrality conditions} \label{subsec:higher}

In this part, we will introduce families of polytopes whose Ehrhart coefficients are always volumes of certain projections of the original polytopes and are hence positive. 

\subsubsection{Cyclic polytopes}
We start with a well-known family of polytopes: \emph{cyclic polytopes}:
The \emph{moment curve} in $\R^d$ is defined by 
\[\nu_d: \R \to \R^d, x \mapsto \nu_d(u) = (u, u^2, \dots, u^d).\] 
Let $U = \{u_1, \dots, u_n \}_<$ be a linear ordered set. Then the \emph{cyclic polytope} $C_d(U) = C_d(u_1, \dots, u_n)$ is the convex hull of $n > d$ distinct points $\nu_d(t_i), 1 \le i \le n,$ on the moment curve:
\[ C_d(U) := \conv \{\nu_d(u_1), \nu_d(t_2), \dots, \nu_d(u_n) \}.\]
Cyclic polytopes form an interesting family of polytopes. For instance, its facets are determined by the Gale evenness condition \cite[Theorem 0.7]{zie}, and the number of $i$-dimensional faces of $C_d(U)$ (where $|U|=n$) is the upper bound for the number of $i$-dimensional faces of all $d$-dimensional polytopes with $n$ vertices \cite{mcmullen1970}. 

The following theorem on the Ehrhart polynomial of integral cyclic polytopes was initially conjectured in \cite{beck-deloera-et2005} by Beck, De Loera, Develin, Pfeifle and Stanley, and then proved in \cite{cyclic} by the author.

\begin{thm}[L.] \label{thm:cyclic}
For any $d$-dimensional integral cyclic polytope $P = C_d(U) \subset \R^d,$ we have that 
\begin{equation}\label{equ:cyclic}
  i(P, t) =  \vol_d (P) t^d + i(\pi(P), t) = \sum_{k=0}^d \vol_{k} \left(\pi^{(d-k)}(P)\right) t^k,
\end{equation}
where $\pi^{(d-k)}:\R^d \to \R^{k}$ is the map that ignores the last $d-k$ coordinates of a point, and $\vol_k(Q)$ is the volume of $Q$ in the $k$-dimensional space $\R^k.$
\end{thm}

The first step of the proof is to reduce the problem to simplices by using triangulations. For the simplex case, we consider the set obtained by removing the lower envelope of $C_d(U)$ (with $|U|=d+1$), and we decompose this set into $d!$ signed (convex) half-open sets $S_\sigma$, each of which corresponds to a permutation $\sigma$ in the symmetric group $\fS_d$. One important feature of this decomposition is that the number of lattice points in each piece $S_\sigma$ can be expressed in a simple formula involving the permutation $\sigma$, which makes it possible to compute the summation of all $d!$ terms.

\subsubsection{$k$-integral polytopes} \label{subsubsec:kintegral}

Since the work in \cite{cyclic}, the author generalized the family of integral cyclic polytopes to a bigger family of integral polytopes, ``lattice-face polytopes'', and showed that their Ehrhart polynomials are also in the simple form of \eqref{equ:cyclic} \cite{lattice-face, note-lattice-face}.
Later in \cite{high-integral}, the author improved her results by introducing a notion of ``higher integrality'', which we will detail below. 

Recall that a lattice point is also called an \emph{integral} point. A point can be considered as a $0$-dimensional affine space. We first extend this concept of integrality to higher dimensional affine spaces: 
An $\ell$-dimensional affine space $W$ in $\R^d$ is \emph{integral} if
\[\pi^{(d-\ell)}(W \cap \Z^d) = \Z^{\ell}.\]
Note that this definition with $\ell=0$ is consistent with the original definition of an integral point.
\begin{ex}[lines in $\R^2$] \label{ex:1integral}
See the left side of Figure \ref{fig:higher} for examples of $1$-dimensional affine space in $\R^2.$ The black lines are integral while the red lines are \emph{not} integral. For the slanted red line, say $L_1,$ we have $\pi^{(2-1)}(L_1 \cap \Z^2) \cong \Z/4\Z.$ For the vertical red line, say $L_2,$ we have $\pi^{(2-1)}(L_2 \cap \Z^2) \cong \Z^0.$
\begin{figure}[htp]
     \scalebox{0.35}{\input{1integral.pstex_t}} \includegraphics[scale=0.55, viewport=90 125 420 315,clip]{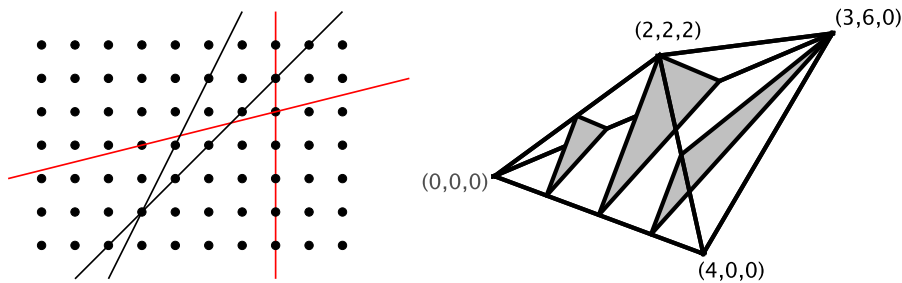}
 \caption{Examples of higher integrality conditions.}
 \label{fig:higher}
 \end{figure}
\end{ex}
Note that in the above example, even though $L_1$ is not integral, after the projection, we still get a $1$-dimensional lattice, which has the same dimension as $L_1.$ In this case, we say $L_1$ is \emph{in general position}. On the contrary, $L_2$ is \emph{not} in general position.

\begin{defn}
Suppose $0 \le k \le d.$ A $d$-dimensional polytope is \emph{$k$-integral} $P$ if for any face $F$ of $P$ of dimension less than or equal to $k$, the affine hull $\aff(F)$ of $F$ is integral.

In particular, when $k = d,$ we call $P$ a \emph{fully integral} polytope. 
\end{defn}

\begin{rem}
	With the above definition, \emph{lattice-face polytopes}, introduced in \cite{lattice-face, note-lattice-face}, can be defined as polytopes that can be triangulated into fully integral simplices, which is a property any (integral) cyclic polytope has. 
Therefore, any cyclic polytope or lattice-polytope is fully integral.
\end{rem}

The main result in \cite{high-integral} is a complete description for the Ehrhart coefficients of a $k$-integral polytope in terms of volumes of projections and Ehrhart polynomials of \emph{slices}. 
\begin{defn}
For any $\y \in \pi^{(d-k)}(P),$ we define \emph{the slice of $P$ over $\y$}, denoted by $\pi_{d-k}(\y, P)$, to be the intersection of $P$ with the inverse image of $\y$ under $\pi^{(d-k)}.$
\end{defn}

Recall that $[t^k]f(t)$ denotes the coefficient of $t^k$ of a polynomial $f(t).$
\begin{thm}[L.]\label{thm:higher}
If $P$ is a $k$-integral polytope, then 
\[ [t^\ell] i(P,t) = 
\begin{cases}
{\vol(\pi^{d-\ell}(P))} & \text{ if $0 \le \ell \le k$,} \\
[t^{\ell-k}] \left(\sum_{\y \in \pi^{(d-k)}(P) \cap \Z^k} i(\pi_{d-k}(\y,P),t) \right) & \text{ if $k+1 \le \ell \le d$}.
\end{cases}\]
Therefore, if $P$ is fully integral, the Ehrhart polynomial of $P$ is in the form of \eqref{equ:cyclic}, and thus $P$ is Ehrhart positive.
\end{thm}
Because both cyclic polytopes and lattice-face polytopes are fully integral polytopes, the above theorem generalizes results in \cite{cyclic, lattice-face, note-lattice-face}.

The following is an example showing how to use Theorem \ref{thm:higher} to obtain the Ehrhart polynomial of a $1$-integral polytope.

\begin{ex}[Example of Theorem \ref{thm:higher}]\label{ex:highpoly}
	Consider the $3$-dimensional polytope 
	\[P = \conv \{(0,0,0), (4,0,0), (3,6,0), (2,2,2)\} \subset \R^3,\] which is illustrated on the right side of Figure \ref{fig:higher}. One checks that $P$ is $1$-integral. Clearly $\pi^{(2)}(P) = [0,4]$ and $\pi^{(3)}(P) = 0.$ By the first part of Theorem \ref{thm:higher}, 
	\[ [t^1]i(P,t) = \vol_1([0,4]) = 4, \quad \text{ and } \quad [t^0] i(P,t) = \vol_0(0) = 1.\]

	For the higher Ehrhart coefficients of $P$, we need to compute the Ehrhart polynomials of slices of $P$ over lattice points in $\pi^{(2)}(P)=[0,4]$. In the picture, the three shaded triangles are the slices of $P$ over the lattice points $1, 2$ and $3.$ The slices of $P$ over lattice points $0$ and $4$ are the single points $(0,0,0)$ and $(4,0,0),$ respectively. We calculate the Ehrhart polynomials of all five slices, 
	by summing which up we obtain $8t^2 + 10t + 5.$ Then the second part of Theorem \ref{thm:higher} says that
	\[ [t^3] i(P,t) = 8 \quad \text{ and } \quad [t^2] i(P,t) = 10.\]
	Therefore, 
	\[ i(P,t) = 8t^3 + 10 t^2 + 4t + 1.\]
\end{ex}

Recall that the \emph{face poset} of a polytope $P$ is the set of all faces of $P$ ordered by inclusion. We say two polytopes have the \emph{same combinatorial type} if they have the same face poset. 
As a byproduct of the study of Ehrhart polynomials of full-integral polytopes, we can also show that Ehrhart positivity is independent from combinatorial types of polytopes \cite{note-lattice-face}.

\begin{thm}[L.]\label{thm:notcomb}
For any rational polytope $P,$ there exists a polytope $P'$ with the same face lattice such that $P'$ satisfies the higher integrality condition and thus is Ehrhart positive. 
\end{thm}

\begin{proof}[Sketch of proof] 
	First, by choosing appropriate bases for our underlying lattice $\Z^d,$ we may assume that the affine hull of any face of $P$ is in general position. 
	
	Next, for any $\s = (s_1, \dots, s_d) \in \Z^d$ and $\x \in \R^d,$ we define 
	\[ \s \star \x = (s_1 x_1, s_2 x_2, \dots, s_d x_d).\]
	So $\s$ is an operator on $\R^d$ that dilates points with different scalars at different coordinates. 
	We observe that for any $\ell$-dimensional affine space $W \subset \R^d$ that is in general position, there exist (positive) integer scalars $c_1, \dots, c_\ell$ such that for any $\s \in \Z_{\neq 0}^d,$ if $c_ms_m$ divides $s_{m+1}$ for each $m \in \{1,2,\dots, \ell\}$, then 
	\[ \s \star W := \{ \s \star \w \ : \ \w \in W \} \] is integral. For example, for the slanted red line $L_1$ appeared in Example \ref{ex:1integral}, one checks that whenever $\s=(s_1, s_2)$ satisfies $4s_1$ divides $s_2$, the affine space $\s \star L_1$ is integral. Hence, we can choose $c_1=4$.

	Since $P$ has finitely many faces, we can apply the above operations inductively on dimensions of faces to obtain a full integral polytope $P'$ that actually defined as $\s \star P$ for some $\s \in \Z_{\neq 0}^d.$
\end{proof}

\begin{rem}
There are a lot of properties of polytopes people study other than Ehrhart positivity, such as ``normality'', ``integer decomposition property'' (or IDP), ``existence of a (regular) unimodular triangulation''. For the majority of them, even if you start with a polytope $P$ that does not have a certain property, dilating $P$ with a large enough scalar often yields a polytope with the desired property (see, for example, \cite{BruGubTru, CoxHaaHibHig, HaaPafPieSan}). 
	Clearly, simple dilations won't change the answer to the Ehrhart positivity question for any polytope. After all, $i(kP, t) = i(P, kt).$ Hence, the Ehrhart coefficients of a dilation of $P$ have exactly the same sign pattern as Ehrhart coefficients of $P.$

However, our proof of Theorem \ref{thm:notcomb} says that dilating in different directions with different scalars can change a non-Ehrhart-positive polytope to a Ehrhart-positive one. 
\end{rem}

\section{McMullen's formula and positivity of generalized permutohedra} \label{sec:mcmullen}

The main purpose of this section is to study the Ehrhart positivity conjecture for generalized permutohedra. After reviewing previously known results supporting this conjecture, we introduce \emph{McMullen's formula}, which is a formula for computing the number of lattice points inside polytopes. This provides us a way of attacking the question of Ehrhart positivity by reducing the problem to ``$\alpha$-positivity''. We then discuss the author's joint work \cite{ehrhartpos-gp-fpsac, BValpha} with Castillo on the Ehrhart positivity conjecture of generalized permutohedra using this approach.

\subsection{Motivation and evidence}

In this part, we discuss the motivation for considering the Ehrhart positivity conjecture of generalized permutohedra and prior work by Postnikov which provides evidence for this conjecture.
We start by formally defining \emph{generalized permutohedra}, the main family of polytopes we study in this section.  
Whenever we talk about generalized permutohedra, we have $D=d+1.$

\subsubsection{Definition and first positivity conjecture}
Given a strictly increasing sequence $\balpha= (\alpha_1,\alpha_2,\cdots,\alpha_{d+1}) \in \R^{d+1}$, 
we define the \emph{usual permutohedron} associated with $\balpha$ as
\[\Perm(\balpha) := \conv\left((\alpha_{\pi(1)},\alpha_{\pi(2)},\cdots, \alpha_{\pi({d+1})}) \ :  \ \pi\in \fS_{d+1}\right)\]
In particular, if $\balpha = (1, 2, \dots, {d+1}),$ we obtain the 
regular permutohedron $\Pi_d$ considered in Example \ref{ex:regular}.
In \cite{post}, Postnikov defined \emph{generalized permutohedra} to be polytopes that can be obtained from usual permutohedra by moving vertices while preserving all edge directions. (Note that in this definition, edges are allowed to degenerate, and hence vertices can collapse.)

In \cite{deloeraHK}, De Loera, Haws, and Koeppe study the Ehrhart polynomials of matroid base polytopes, and conjecture those all have positive coefficients.  
However, 
it turns out that every matroid base polytope is a generalized permutohedron \cite[Section 2]{ardilaBD}.
In \cite{ehrhartpos-gp-fpsac, BValpha}, Castillo and the author generalize the conjecture of De Loera et al. to all integral generalized permutohedra:
\begin{conj}[Castillo-L.]\label{conj:posgp}
All integral generalized permutohedra are Ehrhart positive. 	
\end{conj}

Indeed, due to Postnikov's work, a big family of generalized permutohedra is already known to be Ehrhart positive, which provides a strong evidence to the above conjecture. We describe his work below.

\subsubsection{Ehrhart positivity of type-$\cY$ generalized permutohedra} \label{subsubsec:typey}

In \cite{post}, Postnikov considers Minkowski sums of dilated simplices: For any nonempty subset $I \subseteq [d+1],$ define the simplex
\[ \Delta_I := \conv \{ \be_i \ : \ i \in I\}.\]
Let $\y = ( y_I \ : \ \emptyset \neq I \subseteq [d+1]) \in \left(\R_{\ge 0}\right)^{2^{d+1}-1}$ be a vector indexed by nonempty subsets of $[d+1]$ with nonnegative entries. We define the polytope 
\[ P_d^{\cY} (\y) := \sum_{\emptyset \neq I \subseteq [d+1]} y_I \Delta_I\]
as the Minkowski sum of the simplices $\Delta_I$ dilated by the factor $y_I.$
Postnikov shows that $P_d^{\cY}(\y)$ is always a generalized permutohedron \cite[Proposition 6.3]{post}; however not every generalized permutohedron can be expressed as $P_d^{\cY}(\y)$ for some $\y$ \cite[Remark 6.4]{post}. Therefore, we call $P_d^{\cY}(\y)$ a \emph{type-$\cY$ generalized permutohedron}.

Postnikov then reformulates the construction of $P_d^\cY(\y)$ using bipartite graphs: Let $G$ be a subgraph of the bipartite graph $K_{c,d+1}$ without isolated vertices. Label the vertices of $G$ on the left by $l_1, l_2, \dots, l_c$ and vertices on the right by $r_1, r_2, \dots, r_{d+1}.$ For each $1 \le j \le c,$ we let 
\[ I_j^G = \{ i \in [d+1] \ : \ \{l_j, r_i\} \text{ is an edge of $G$}\}.\]
For any $(y_1, y_2, \dots, y_c) \in \left(\R_{\ge 0}\right)^c,$ we define the polytope
\[ P_G(y_1, \dots, y_c) := \sum_{j=1}^c y_j \Delta_{I_j}^G.\]

\begin{rem}\label{rem:typey}
	It is clear that $P_G(y_1, y_2, \dots, y_c)$ is the type-$\cY$ generalized permutohedron $P_d^\cY(\y)$ where $y_I = \sum_{j: I_j=I} y_j.$ Conversely, the type-$\cY$ generalized permutohedron $P_d^\cY(\y)$ is the polytope $P_G(\y)$ where $G$ is the subgraph of $K_{2^{d+1}-1,d+1}$ such that left vertices of $G$ are indexed by nonempty subsets $I$ of $[d+1],$ and the left vertex $l_I$ is adjacent to the right vertex $r_i$ if and only if $i \in I.$
\end{rem}

In \cite[Section 11]{post}, Postnikov defines the ``trimmed generalized permutohedron'' as the ``Minkowski difference'' of $P_G(y_1, \dots, y_c)$ and the simplex $\Delta_{[d+1]}.$
By providing a formula for the number of lattice points in a trimmed generalized permutohedron, he obtains a formula for the number of lattice points in $P_G(y_1, \dots, y_c)$ \cite[Theorem 11.3]{post}, which leads to an expression for the Ehrhart polynomial of $P_G(y_1, \dots, y_c)$ as a summation over \emph{$G$-draconian sequences}.
\begin{defn}[Definition 9.2 in \cite{post}]
	A sequence of nonnegative integers $\bg =(g_1, g_2, \dots, g_c)$ is a \emph{$G$-draconian sequence} if $\sum_{j=1}^c g_j = d$ and for any subset $\{j_1, \dots, j_k\} \subseteq [c],$ we have $|I_{j_1}^G \cup \cdots \cup I_{j_k}^G| \ge g_{j_1} + \cdots + g_{j_k} +1.$ 
\end{defn}

\begin{thm}[Postnikov]\label{thm:typey}  Suppose $G$ is a subgraph of $K_{c,d+1}$ without isolated vertices such that $I_1^G= [d+1].$ Let $y_1, \dots,y_c \in \N.$
	Then the Ehrhart polynomial of $P_G(y_1, \dots, y_c)$ is given by
	\[ i(P_G(y_1,y_2,\dots, y_c), t) = \sum_{\bg} \multiset{y_1t+1}{g_1} \prod_{k=2}^c \multiset{y_k t}{g_k},\]
	where the summation is over all $G$-draconian sequences $\bg = (g_1, \dots, g_c).$
\end{thm}

Similar to the results discussed in \S \ref{subsubsec:PitSta} and \S \ref{subsubsec:flow}, it follows from Lemma \ref{lem:product}/\eqref{item:sum} that the Ehrhart polynomial described in the above theorem has positive coefficients. Thus, by Remark \ref{rem:typey}, we immediately have the following:
\begin{cor}\label{cor:typey}
Any integral type-$\cY$ generalized permutohedron is Ehrhart positive.
\end{cor}
Note that as we pointed out above, type-$\cY$ generalized permutohedra do not contain all generalized permutohedra. Thus, Conjecture \ref{conj:posgp} does not follow from the above result.
\begin{ex}[Pitman-Stanley polytopes again]
	Let $G$ be a subgraph of $K_{d+1, d+1}$ where for each $j \in [d+1],$ the left vertex $l_j$ is adjacent to right vertices $r_j, r_{j+1}, \dots, r_{d+1}.$ Then for any $\y = (y_1, \dots, y_{d+1}) \in \N^{d+1},$ 
	\[ P_G(\y) = P_G(y_1, \dots, y_{d+1}) = \sum_{j=1}^{d+1} y_j \Delta_{[j,d+1]},\]
	where $[j,d+1]=\{j, j+1, \dots, d+1\}.$
	It follows from Proposition 6.3 of \cite{post} that the inequality description of this polytope is 
	\small
	\[ P_G(\y) = \left\{ \x \in \R^{d+1} \ : \  x_i \ge 0 \text{ and } \sum_{j=1}^{i} x_j \le \sum_{j=1}^{i} y_j \text{ for $i = 1, 2,\dots, d-1$, and } \sum_{j=1}^{d+1} x_j = \sum_{j=1}^{d+1} y_j \right\}.\]
\normalsize
It is easy to see the map $\pi: \R^{d+1} \to \R^d$ that ignores the last coordinate of a point induces a unimodular transformation from $P_G(\y)$ to the Pitman-Stanley polytope $\PS_d(\widehat{\y})$ considered in \S \ref{subsubsec:PitSta}, where $\widehat{\y} = (y_1,y_2, \dots, y_d).$

One can also check that the $G$-draconian sequences for the graph $G$ given in this example are those $\bg = (g_1, \dots, g_{d+1}) \in \N^{d+1}$ satisfying
\[ \sum_{j=1}^d g_j = d, \ g_{d+1}=0, \text{ and } \sum_{j=1}^k g_j \ge k \text{ for $k=1,2, \dots, d-1$}.\]
Hence, it can be verified that Theorem \ref{thm:pitsta} is a special case of Theorem \ref{thm:typey}. 
\end{ex}

The family of type-$\cY$ generalized permutohedra not only includes the Pitman-Stanley polytope as we have seen in the example above, but also includes associahedra, cyclohedra, and more (see \cite[Section 8]{post}). However, it follows from work by Ardila, Benedetti and Doker that type-$\cY$ generalized permutohedra do not contain all matroid base polytopes \cite[Proposition 2.3 and Example 2.6]{ardilaBD}.
Therefore, Corollary \ref{cor:typey} does not settle either Conjecture \ref{conj:posgp} or the Ehrhart positivity conjecture on matroid base polytopes by De Loera et al \cite{deloeraHK}.

\subsection{McMullen's formula, $\alpha$-positivity, and a reduction theorem} \label{subsec:mcmullen}

The goal of this part is to introduce McMullen's formula and discuss why it is a good tool to show Ehrhart positivity of a family of polytopes constructed from a fixed projective fan when an $\alpha$-construction satisfies certain valuation properties. (We will discuss in \S \ref{subsec:pos-gp} that generalized permutohedra form a family of polytopes constructed from the Braid fan. Hence, the techniques introduced here are relevant to our question.)

Throughout the rest of this section, we let $V$ be a subspace of $\R^D$ and $V^*$ be the dual space of $V.$ For any polytope $P$, we use the notation $\lin(P)$ to denote the linear space obtained by shifting the affine span $\aff(P)$ to the origin.

\subsubsection{Cones} 
We need the concepts of \emph{cones}, particularly \emph{feasible cones} and \emph{normal cones}, before we start our discussion. 

A \emph{(polyhedral) cone} is the set of all nonnegative linear combinations of a finite set of vectors. 
A cone is \emph{pointed} if it does not contain a line.
\begin{defn}
  Suppose 
  $P$ is a polytope satisfying $\lin(P) \subseteq V.$ 
  \begin{enumerate}[(i)]
\item The \emph{feasible cone} of $P$ at $F$ is:
\[
{\fcone(F,P)} := \left\{ \u \in V : x + \delta \u \in P \hspace{5pt}\text{for sufficiently small $\delta$}\right\},
\]
where $x$ is any relative interior point of $F.$  (It can be checked that the definition is independent of the choice of $x.$) 

The \emph{pointed feasible cone} of $P$ at $F$ is 
\[ \fcone^p(F,P) = \fcone(F,P)/\lin(F).\]
 
 \item 
Given any face $F$ of $P$, the \emph{normal cone} of $P$ at $F$ with respect to $V$ is 
\[
\ncone_V(F, P):=\left\{ \u \in V^* \ : \ \langle \u, \bp_1 \rangle \geq \langle \u, \bp_2 \rangle, \quad \forall \bp_1\in F,\quad \forall \bp_2\in P \right\}.
\]
Therefore, $\ncone_V(F,P)$ is the collection of linear functionals $\u$ in $V^*$ such that $\u$ attains maximum value at $F$ over all points in $P.$

The \emph{normal fan} $\Sigma_V(P)$ of $P$ with respect to $V$ is the collection of all normal cones of $P$. 

 \end{enumerate}
\end{defn}
Normal cones and pointed feasible cones are related by polarity.
\begin{defn} 
	Let $K \subseteq V^*$ be a cone, and let $W$ be the subspace of $V^*$ spanned by $K.$ (So $W^*$ is a quotient space of $V$.) The \emph{polar cone} of $K$ is the cone 
	\[ K^\circ = \{ \y \in W^* \ : \ \langle \x, \y \rangle \le 0,  \ \forall \x \in K\}.\]
\end{defn}

\begin{lem}[Lemma 2.4 of \cite{BValpha}] Suppose $P$ is a polytope satisfying $\lin(P) \subseteq V$ and $F$ is a face of $P.$ 
	Then $(\ncone_V(F,P))^\circ$ is a pointed cone, and is invariant under the choice of $V$. So we may omit the subscript $V$ and just write $(\ncone(F,P))^\circ$. Furthermore,
\[ \ncone(F,P)^\circ = \fcone^p(F,P).\]
\end{lem}

\subsubsection{McMullen's formula and a refinement of positivity}
In 1975 Danilov asked, in the context of toric varieties, whether it is possible to 
construct a function $\alpha$ such that for any integral polytope $P$, we have
\begin{equation}\label{equ:exterior}
	|P \cap \Z^D|= \displaystyle \sum_{F: \textrm{ a face of $P$}} \alpha(F,P) \ \nvol(F) ,
\end{equation}
where $\alpha(F,P)$ depends only on the normal cone of $P$ at $F,$ and $\nvol(F)$ is the volume of $F$ normalized to the lattice $\aff(F) \cap \Z^D.$

McMullen \cite{mcmullen1993} was the first to confirm the existence of \eqref{equ:exterior} in a non-constructive way. Hence, we refer to the above formula as \emph{McMullen's formula}. 
Pommersheim and Thomas \cite{pommersheim-thomas} provide a canonical construction based on choices of flags. 
Berline and Vergne \cite{berline-vergne} give a construction in a computable way. 
Most recently, Ring-Sch\"urmann \cite{RinSch} give another construction for McMullen's formula based on a choice of fundamental cells.

Before discussing a specific construction, even the existence of McMullen's formula has interesting consequences. In fact, it was one of the ingredients used in proving the results on higher integrality conditions discussed in \S \ref{subsubsec:kintegral}. More importantly, it provides another proof for Ehrhart's theorem (Theorem \ref{thm:ehrhart}) as well as a refinement of Ehrhart positivity. Note that pointed feasible cones do not change when we dilate a polytope. Thus, applying McMullen's formula to $tP$ and rearranging coefficients, we obtain a formula for the function $i(P,t):$
\[ i(P,t) =\sum_{k=0}^{\dim P} \left(\sum_{F: \text{$k$-dimensional face of $P$}} \alpha(P, F) \nvol(F) \right) \cdot t^k.\] 
Hence, $i(P,t)$ is a polynomial in $t$ of degree $\dim P,$ and
the coefficient of $t^k$ in $i(P,t)$ is given by 
\begin{equation}\label{equ:coeff}
	[t^k] i(P,t) = \sum_{F: \text{$k$-dimensional face of $P$}} \alpha(P, F) \nvol(F).
\end{equation}

\begin{ex} Setting $k=0$ in \eqref{equ:coeff}, we obtain
	\[ [t^0] i(P,t) = \sum_{v: \text{vertex of $P$}} \alpha(P, v) \nvol(v).\]
	Note that $[t^0] i(P,t)$ is the constant term of the Ehrhart polynomial of $P,$ which is known to be $1$ for any integral polytope $P.$ Furthermore, the normalized volume of any vertex is $1.$ Hence, the above equation becomes
	\[  \sum_{v: \text{vertex of $P$}} \alpha(P, v) =1.\]
	See Figure \ref{fig:diffalpha} for $\alpha$-values of the vertices of the triangle $P = \conv( (0,0), (2,0), (2,1))$ arising from different constructions.
\begin{figure}[h]
\begin{center}
\begin{tikzpicture}

	\begin{scope}[yshift=-2cm]
		\node at (-4,1.5){\footnotesize Pommersheim-Thomas}; 
\draw[fill=black] (-5,2) circle[radius=0.5mm];
\draw[fill=black] (-4,2) circle[radius=0.5mm];
\draw[fill=black] (-3,2) circle[radius=0.5mm];
\draw[fill=black] (-3,3) circle[radius=0.5mm];
\draw (-5,2) -- (-4,2) -- (-3,2) --  (-3,3) -- cycle;

\node[left] at (-5,2){$\frac{5}{12}$};
\node[right] at (-3,2){$\frac{5}{12}$};
\node[right] at (-3,3){$\frac{1}{6}$};
\end{scope}

\begin{scope}[xshift=4cm]
	\node at (-4,-0.5){\footnotesize{Berline-Vergne}}; 
\draw[fill=black] (-5,0) circle[radius=0.5mm];
\draw[fill=black] (-4,0) circle[radius=0.5mm];
\draw[fill=black] (-3,0) circle[radius=0.5mm];
\draw[fill=black] (-3,1) circle[radius=0.5mm];
\draw (-5,0) -- (-4,0) -- (-3,0) --  (-3,1) -- cycle;

\node[left] at (-5,0){$\frac{9}{20}$};
\node[right] at (-3,0){$\frac{1}{4}$};
\node[right] at (-3,1){$\frac{3}{10}$};
\end{scope}

\begin{scope}[xshift=8cm, yshift=2cm]
	\node at (-4,-2.5){\footnotesize{Ring-Sch\"urmann}}; 
\draw[fill=black] (-5,-2) circle[radius=0.5mm];
\draw[fill=black] (-4,-2) circle[radius=0.5mm];
\draw[fill=black] (-3,-2) circle[radius=0.5mm];
\draw[fill=black] (-3,-1) circle[radius=0.5mm];
\draw (-5,-2) -- (-4,-2) -- (-3,-2) --  (-3,-1) -- cycle;

\node[left] at (-5,-2){$\frac{1}{2}$};
\node[right] at (-3,-2){$\frac{1}{4}$};
\node[right] at (-3,-1){$\frac{1}{4}$};
\end{scope}

\end{tikzpicture}
\end{center}
\caption{Different $\alpha$-constructions.}
\label{fig:diffalpha}
\end{figure}
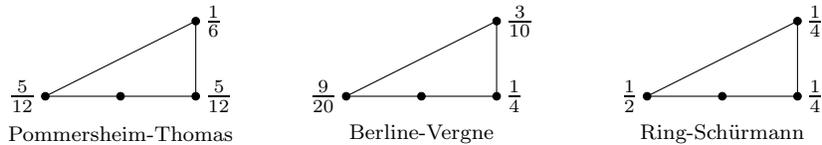
\end{ex}

Since $\nvol(F)$ is a positive number, 
it follows from \eqref{equ:coeff} that $\alpha$-values refine Ehrhart coefficients. We say a polytope $P$ is \emph{$\alpha$-positive for $k$-faces} if $\alpha(F, P)$ is positive for all $k$-dimensional faces $F$ of $P,$ and say $P$ is \emph{$\alpha$-positive} if all $\alpha$'s associated to $P$ are positive.  
The following result immediately follows from expression \eqref{equ:coeff}.
\begin{lem}\label{lem:red1}
Suppose $\alpha$ is a solution to McMullen's formula. Let $P$ be an integral polytope. For a fixed $k,$ 
if $P$ is $\alpha$-positive for $k$-faces, then the coefficient of $t^k$ in the Ehrhart polynomial $i(P,t)$ of $P$ is positive.

Hence, 
if $P$ is $\alpha$-positive, then $P$ is Ehrhart positive.
\end{lem}

\subsubsection{BV-construction and the Reduction Theorem} \label{subsubsec:bvred} At the first glance, $\alpha$-positivity, being a refinement of Ehrhart-positivity, is a more difficult question to consider. However, for $\alpha$-constructions that satisfy certain properties, studying $\alpha$-positivity instead does not necessarily make the problem harder. Berline and Vergne \cite{berline-vergne} give such an $\alpha$-construction, of which we give a quick review below.
Recall that the \emph{indicator function} of a set $A \subseteq V$ is the function $[A]: V \to \R$ defined as $[A](x) = 1$ if $x \in A$, and $[A](x)=0$ if $x \not\in A.$ 
The \emph{algebra of rational cones}, denoted by $\cC(V)$, is the vector space over $\Q$ spanned by the indicator functions $[C]$ of all rational cones $C \subset V.$ We consider $\cC(V)$ a subspace of the vector space of all functions on $V.$ Hence, in general, the indicators $[C]$ of rational cones do not form a basis of $\cC(V)$ since there are many relations among them. 
%
\begin{thm}[Berline-Vergne] 
	There exists a function $\Psi$ from the set of indicator functions $[C]$ of rational cones $C$ in $V$ to $\R$ with the following properties: 

\begin{enumerate}[(P1)]
	\item $\Psi$ induces a \emph{valuation} on the algebra of rational cones in $V$, i.e., $\Psi$ induces a linear transformation from $\cC(V)$ to $\R.$ 
 
  \item If a cone $C$ contains a line, then $\Psi([C])=0$.

  \item 
	  $\Psi$ is invariant under orthogonal unimodular transformation, thus, is \emph{symmetric about coordinates}, that is, invariant under rearranging coordinates with signs.
  \item Setting \begin{equation}
	\alpha(F,P) :=\Psi([\fcone^p(F,P)]),
	\label{equ:alphapsi}
\end{equation}
gives a solution to McMullen's formula. 
\end{enumerate}
\end{thm}
We refer to Berline-Vergne's construction of $\Psi$ and $\alpha$ as \emph{BV-construction} and \emph{BV-$\alpha$-valuation}, respectively. If $\alpha$ is the BV-$\alpha$-valuation, we use the terminology \emph{BV-$\alpha$-positivity} instead of $\alpha$-positivity. 

Properties (P1) and (P2) are the ``certain valuation properties'' we mentioned at the beginning of \S \ref{subsec:mcmullen}. The following Reduction Theorem lays out a consequence of these two properties. 

\begin{thm}[Castillo-L., Reduction Theorem \cite{BValpha}] \label{thm:reduction-gen}
	Suppose $\Psi$ is a function from the set of indicator functions of rational cones $C$ in $V$ to $\R$ such that properties (P1) and (P2) hold, and suppose $\alpha$ is defined as in \eqref{equ:alphapsi}.

Let $P$ and $Q$ be two polytopes such that $\lin(P)$ and $\lin(Q)$ are both subspaces of $V.$
Assume the normal fan $\Sigma_V(P)$ of $P$ with respect to $V$ is a refinement of the normal fan $\Sigma_V(Q)$ of $Q$ with respect to $V$.

Then for any fixed $k,$ 
if $P$ is $\alpha$-positive for $k$-faces, then $Q$ is $\alpha$-positive for $k$-faces.
\end{thm}

One important implication of the Reduction Theorem is that we can reduce the problem of $\alpha$-positivity of a family of polytopes constructed from a fan to the problem of $\alpha$-positivity of a single polytope in the family.

\begin{defn}\label{defn:normalfamily} Let $\Sigma$ be a \emph{projective} fan in $V^*$, i.e., a fan that is a normal fan of some polytope. Let $\Poly(\Sigma)$ be the set of polytopes $Q$ whose normal fan $\Sigma_V(Q)$ with respect to $V$ coarsens $\Sigma.$ 
\end{defn}

\begin{cor}\label{cor:reduction-gen}
	Assume the hypothesis on $\Psi$ and $\alpha$ in Theorem \ref{thm:reduction-gen}. Let $\Sigma$ be a projective fan in $V^*$, and let $P$ be a polytope such that $\Sigma_V(P) = \Sigma.$
	Then $\alpha$-positivity (for $k$-faces) of $P$ implies $\alpha$-positivity (for $k$-faces) of $Q$ for any $Q \in \Poly(\Sigma).$
	
	Assume further that $\alpha$ is a solution to McMullen's formula.
Then for any integral polytope $Q \in \Poly(\Sigma)$, $\alpha$-positivity for $k$-faces of $P$ implies the coefficient of $t^k$ in $i(Q,t)$ is positive. Hence, $\alpha$-positivity of $P$ implies Ehrhart-positivity of $Q$. 
\end{cor}

\begin{proof} The first part follows directly from the Reduction Theorem, and the second assertion follows from the first part and Lemma \ref{lem:red1}.
\end{proof}

Therefore, even though proving $\alpha$-positivity is more difficult than proving Ehrhart-positivity for an individual polytope, it could be easier if we consider a family of polytopes $\Poly(\Sigma)$ associated to a fixed projective fan $\Sigma$, as we only need to prove $\alpha$-positivity for one polytope in the family.
Finally, because the BV-construction satisfies properties (P1), (P2) and (P4), all the results discussed above apply to the BV-construction or the BV-$\alpha$-valuation. 
These ideas are illustrated by Example \ref{ex:for-red-cor} below.




\subsection{Positivity of generalized permutohedra}\label{subsec:pos-gp}

In this part, 
we apply the Reduction Theorem to reduce our first conjecture --- Conjecture \ref{conj:posgp} --- to a conjecture on $\alpha$-positivity of regular permutohedra.  
Then we report partial progress made on both conjectures by using McMullen's formula with BV-$\alpha$-valuation \cite{ehrhartpos-gp-fpsac, BValpha}.

\subsubsection{Second positivity conjecture}
Postnikov, Reiner, and Williams give several equivalent definitions for generalized permutohedra, one of which uses concepts of normal fans \cite[Proposition 3.2]{PosReiWil}.
Recall that the \emph{Braid fan}, denoted by $\Br_d,$ is the complete fan in $\R^{d+1}$ given by the hyperplanes $x_i - x_j = 0$ for all $i\neq j$.

\begin{prop}[Postnikov-Reiner-Williams] \label{prop:coarser}
	A polytope $P$ in $V=\R^{d+1}$ is a generalized permutohedron if and only if its normal fan $\Sigma_V(P)$ with respect to $V$ is refined by the Braid fan $\Br_d$.
\end{prop}
Using the notation we give in Definition \ref{defn:normalfamily}, the above result precisely says that the family of generalized permutohedra in $\R^{d+1}$ is $\Poly(\Br_d).$
Furthermore, it follows from \cite[Proposition 2.6]{post} that any usual permutohedron in $\R^{d+1}$ has the Braid fan $\Br_d$ as its normal fan. 
In particular, the normal fan of the regular permutohedron $\Pi_d$ is $\Br_d.$
In \cite{BValpha}, Castillo and the author use these results together with the Reduction theorem and its consequence (i.e., Corollary \ref{cor:reduction-gen}) to reduce Conjecture \ref{conj:posgp} to the following conjecture:
\begin{conj}[Castillo-L.]\label{conj:alphas}
  Every regular permutohedron $\Pi_d$ is BV-$\alpha$-positive.
\end{conj}

The following example demonstrates how Corollary \ref{cor:reduction-gen} works and why Conjecture \ref{conj:posgp} can be reduced to Conjecture \ref{conj:alphas}.
\begin{ex}\label{ex:for-red-cor}
	Let $P, Q_1, Q_2$ and $Q_3$ be the $2$-dimensional polytopes together with their normal fans shown in Figure \ref{fig:red-ex}. 
One notices that $P$ is the regular permutohedron $\Pi_2$ whose normal fan is $\Br_2$, and each $Q_i$ is a generalized permutohedron whose normal fan coarsens $\Br_2$.

	All the BV-$\alpha$-values of the six vertices of $P$ are $1/6.$ Since $Q_1$ has the same normal fan as $P$, all of its six vertices also have the same BV-$\alpha$-values. Now the normal fan of $Q_2$ coarsens that of $P.$ In particular, if we let $v$ be the vertex on the bottom-left of $Q_2,$ then the normal cone $\ncone(v,Q_2)$ of $Q_2$ at $v$ is the union of the normal cones of $P$ at two of its vertices. It is a consequence of the ``valuation properties'' (P1) and (P2) that the BV-$\alpha$-values $\alpha(v,Q_2)$ is the sum of the BV-$\alpha$-values of these two vertices of $P.$ Therefore, as shown in the figure, $\ds \alpha(v,Q_2) = 1/6 + 1/6 = 1/3.$ One sees that similar phenomenon happens for the polytope $Q_3.$

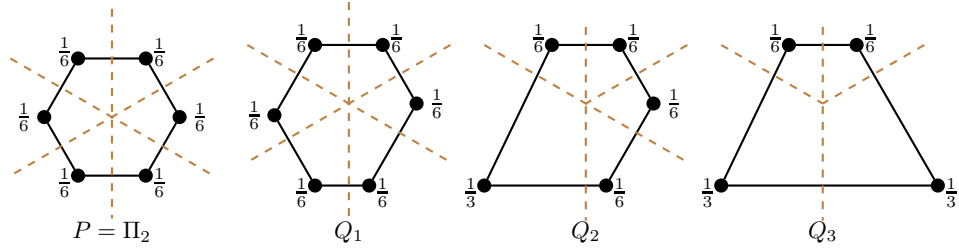
\begin{figure}
	\begin{center}
		\begin{tikzpicture}[thick, scale=0.9, baseline=80]
	  \begin{scope}[xshift=-5.5cm, yshift=-0.2cm] 
  \tikzstyle{every node}=[circle, fill, draw,inner sep=2pt, scale=0.8]
\node (a) at (-1,1) {};
\node (b) at (-0.5,0.134) {};
\node (c) at (0.5,0.134) {};
\node (d) at (1,1) {};
\node (e) at (0.5,1.866) {};
\node (f) at (-0.5,1.866) {};

\tikzstyle{every node}=[inner sep=1pt, minimum width=14pt,scale=0.7]

\draw [thick] (a) -- (b) -- (c) -- (d) -- (e) --(f) -- (a);
\draw [thick, dashed, brown] (1.5, 1.866) -- (-1.5, 0.134);
\draw [thick, dashed, brown] (-1.5, 1.866) -- (1.5, 0.134);
\draw [thick, dashed, brown] (0,2.6) -- (0,0.5-1);

\tikzstyle{every node}=[inner sep=1pt, minimum width=14pt,scale=0.9]
  \node at (-1.3,1) {$\frac{1}{6}$};
  \node at (-0.7,0.034) {$\frac{1}{6}$};
  \node at (0.7,0.034) {$\frac{1}{6}$};
  \node at (1.3,1) {$\frac{1}{6}$};
  \node at (0.7,1.966) {$\frac{1}{6}$};
  \node at (-0.7,1.966) {$\frac{1}{6}$};

  \node at (0,-0.7) {$P = \Pi_2$};
\end{scope}

\begin{scope}[xshift=-2cm] 
  \tikzstyle{every node}=[circle, fill, draw,inner sep=2pt, scale=0.8]
\node (a) at (-1.1,0.8268) {};
\node (b) at (-0.5,-0.2124) {};
\node (c) at (0.3,-0.2124) {};
\node (d) at (1,1) {};
\node (e) at (0.5,1.866) {};
\node (f) at (-0.5,1.866) {};

\draw [thick] (a) -- (b) -- (c) -- (d) -- (e) --(f) -- (a);

\draw [thick, dashed, brown] (1.5, 1.866) -- (-1.5, 0.134);
\draw [thick, dashed, brown] (-1.5, 1.866) -- (1.5, 0.134);
\draw [thick, dashed, brown] (0,2.5) -- (0,0.268-1);

\tikzstyle{every node}=[inner sep=1pt, minimum width=14pt,scale=0.9]
\node at (-1.4,0.8268) {$\frac{1}{6}$};
\node at (-0.7,-0.3124) {$\frac{1}{6}$};
\node at (0.5,-0.3124) {$\frac{1}{6}$};
\node at (1.3,1) {$\frac{1}{6}$};
\node at (0.7,1.966) {$\frac{1}{6}$};
\node at (-0.7,1.966) {$\frac{1}{6}$};

  \node at (0,-0.9) {$Q_1$};
\end{scope}

\begin{scope}[xshift=1.5cm]
  \tikzstyle{every node}=[circle, fill, draw,inner sep=2pt, scale=0.8]
\node (b) at (-1.5,-0.2124) {};
\node (c) at (0.3,-0.2124) {};
\node (d) at (1,1) {};
\node (e) at (0.5,1.866) {};
\node (f) at (-0.5,1.866) {};

\draw [thick] (b) -- (c) -- (d) -- (e) --(f) -- (b);

\draw [thick, dashed, brown] (1.5, 1.866) -- (0,1);
\draw [thick, dashed, brown] (-1.5, 1.866) -- (1.5, 0.134);
\draw [thick, dashed, brown] (0,2.132) -- (0,0.268-1);

\tikzstyle{every node}=[inner sep=1pt, minimum width=14pt,scale=0.9]
\node at (-1.7,-0.3124) {$\frac{1}{3}$};
\node at (0.5,-0.3124) {$\frac{1}{6}$};
\node at (1.3,1) {$\frac{1}{6}$};
\node at (0.7,1.966) {$\frac{1}{6}$};
\node at (-0.7,1.966) {$\frac{1}{6}$};

  \node at (0,-0.9) {$Q_2$};
\end{scope}

\begin{scope}[xshift=5cm] 
  \tikzstyle{every node}=[circle, fill, draw,inner sep=2pt, scale=0.8]
\node (b) at (-1.5,-0.2124) {};
\node (c) at (1.7,-0.2124) {};
\node (e) at (0.5,1.866) {};
\node (f) at (-0.5,1.866) {};

\draw [thick] (b) -- (c) -- (e) --(f) -- (b);

\draw [thick, dashed, brown] (1.5, 1.866) -- (0,1);
\draw [thick, dashed, brown] (-1.5, 1.866) -- (0,1);
\draw [thick, dashed, brown] (0,2.132) -- (0,0.268-1);

\tikzstyle{every node}=[inner sep=1pt, minimum width=14pt,scale=0.9]
\node at (-1.7,-0.3124) {$\frac{1}{3}$};
\node at (1.9,-0.3124) {$\frac{1}{3}$};
\node at (0.7,1.966) {$\frac{1}{6}$};
\node at (-0.7,1.966) {$\frac{1}{6}$};

  \node at (0,-0.9) {$Q_3$};
\end{scope}

  \end{tikzpicture}
\end{center}
\caption{Examples for Corollary \ref{cor:reduction-gen}.}
\label{fig:red-ex}
\end{figure}

The above discussion shows that even if we did not know the BV-$\alpha$-values of vertices of $P,$ because each BV-$\alpha$-value arising from $Q_i$ is a summation of a subset of BV-$\alpha$-values of vertices of $P,$ BV-$\alpha$-positivity of vertices of the regular permutohedron $P=\Pi_2$ would imply BV-$\alpha$-positivity of vertices of the generalized permutohedron $Q_i,$ and thus would imply the constant Ehrhart coefficient of $Q_i$ is positive. 
\end{ex}

Conjecture \ref{conj:alphas} was the main conjecture studied in \cite{BValpha}, and partial progress was made on proving it, which gave us corresponding partial results on Conjecture \ref{conj:posgp}.

\subsubsection{Partial results}


The first approach of attacking Conjecture \ref{conj:alphas} is to directly compute BV-$\alpha$-valuations. In order to do that, we need to compute the BV-construction $\Psi.$ 
One obvious benefit of considering Conjecture \ref{conj:alphas} instead of Conjecture \ref{conj:posgp} is that in each dimension there is only one regular permutohedron, and thus there are a limited number of BV-$\alpha$-values or $\Psi$-values to be computed, especially for small $d.$ 
Therefore, by explicit computation, we obtain the following result.
\begin{thm}[Castillo-L.]\label{thm:truelowdim}
  For $d \le 6,$ the regular permutohedron $\Pi_d$ is BV-$\alpha$-positive. 
Therefore, all the integral generalized permutohedra (including matroid base polytopes) of dimension at most $6$ are Ehrhart positive. 
\end{thm}

Next, instead of focusing on all the coefficients of Ehrhart polynomials, we study certain special coefficients. 
Note that the first, second and last Ehrhart coefficients are always positive, so we only consider other Ehrhart coefficients. Correspondingly, we need to know how to compute the BV-construction $\Psi(C)$ of cones $C$ of dimension $2, 3, \dots, d-1.$
The computation of the function $\Psi$ is carried out recursively. Hence, it is quicker to compute $\Psi$ for lower dimensional cones. 
As a result, the value of $\alpha(F,P)$ is easier to compute if $F$ is a higher dimensional face.

In general, the computation of $\Psi(C)$ is quite complicated. However, when $C$ is a unimodular cone 
computations are greatly simplified. 
In dimensions $2$ and $3$, with the help of Maple code provided by Berline and Vergne, simple closed expression for $\Psi$ of unimodular cones can be obtained \cite[Lemmas 3.9 and 3.10]{BValpha}. Applying these formulas to $\Pi_d,$ we obtain the second partial result towards Conjectures \ref{conj:alphas} and \ref{conj:posgp}:

\begin{thm}[Castillo-L.]\label{thm:34coeff}
  For any $d,$ and any face $F$ of $\Pi_d$ of codimension $2$ or $3,$ we have $\alpha(F, \Pi_d)$ is positive, where $\alpha$ is the BV-$\alpha$-valuation.

Hence, the third and fourth Ehrhart coefficients of any integral generalized permutohedron (including matroid base polytopes) are positive. 
\end{thm}

Finally, the last partial result presented in \cite{BValpha} is the following:
\begin{lem}[Castillo-L.] \label{lem:edge}
  For any $d \le 500,$ and any edge $E$ of $\Pi_d$,  we have $\alpha(E, \Pi_d)$ is positive, where $\alpha$ is the BV-$\alpha$-valuation.

Hence, the linear Ehrhart coefficient of any integral generalized permutohedron (including matroid base polytopes) of dimension at most $500$ is positive. 
\end{lem}

As we have discussed above, in order to compute the BV-$\alpha$-values for an edge of a $d$-dimensional polytope, we have to compute the $\Psi$-value of a $(d-1)$-dimensional cone, which is extremely difficult for large $d$ if we use Berline-Vergne's algorithm directly. 
Therefore, we use a completely different strategy. 
Recall Property (P3) of the BV-construction, which says that $\Psi$ is symmetric about coordinates. Note that the regular permutohedron $\Pi_d$ is a polytope with much symmetry. So a lot of BV-$\alpha$-values of $\Pi_d$ coincide. In particular, we can separate edges of $\Pi_d$ into to $\left\lceil \frac{d}{2} \right\rceil$ groups, where edges in each group share the same BV-$\alpha$-values.

\begin{proof}[Idea of Proof for Lemma \ref{lem:edge}]
	If we know the $\alpha$-values for a give polytope $P,$ Equation \eqref{equ:coeff} gives us a way to compute the Ehrhart coefficients. However, we can also use \eqref{equ:coeff} in the other direction: Suppose we know the linear coefficient of $i(P,t),$ Equation \eqref{equ:coeff} gives us an equation for $\alpha$-values arising from edges of $P:$
	\[ \sum_{E: \text{edge of $P$}} \alpha(E, P) \nvol(E) = [t^1] i(P,t).\]

The $\alpha$-values for the regular permutohedron also appear in other generalized permutohedra as all of them are in the family $\Poly(\Br_d).$ Therefore, if we can find $\left\lceil \frac{d}{2} \right\rceil$ ``independent'' generalized permutohedra for which we know their linear Ehrhart coefficients, then we can set up a $\left\lceil \frac{d}{2} \right\rceil \times \left\lceil \frac{d}{2} \right\rceil$ linear system for the $\left\lceil \frac{d}{2} \right\rceil$ $\alpha$-values arising from edges of $\Pi_d.$ Solving the system, we obtain all these $\alpha$-values.
See \cite[Example 3.15]{BValpha} for an example of how we can solve a linear system to find $\alpha$'s.

Recall that Postnikov gives explicit formulas for the Ehrhart polynomials of type-$\cY$ generalized permutohedra (see Theorem \ref{thm:typey}). Among all the non-trivial Ehrhart coefficients, the linear terms can be easily described. 
Using these, we were able to set up, for each $d,$ a desired linear system which is actually triangular. 
Solving the system for $d \le 500,$ we confirmed positivity of all $\left\lceil \frac{d}{2} \right\rceil$ $\alpha$'s arising from edges of $\Pi_d.$ 
\end{proof}




\subsection*{Equivalence Statements} In addition to the partial results discussed above, two equivalent statements to Conjecture \ref{conj:alphas} were discovered. The first states that Conjecture \ref{conj:alphas} holds if and only if the mixed lattice point valuation on hypersimplices is positive \cite[Corollary 5.6]{BValpha}.

The second equivalent statement is in terms of Todd classes. The BV-construction gives one way to write the Todd class of the permutohedral variety in terms of the toric invariant cycles. We can show that if there is \emph{any} way of writing such class as a positive combination of such cycles, then the BV-$\alpha$-valuation is one of them. (See \cite[Proposition 7.2]{bvpos-gp} or \cite{castillo-thesis}.)

\section{Negative Results} \label{sec:negative}

\commentout{In this section, we will discuss examples and constructions of polytopes with negative middle Ehrhart coefficients. 
In particular, as consequences of examples studied in \S \ref{subsec:smooth} and \S \ref{subsec:stanleyex}, we provide in \S \ref{subsec:nonpositive} a list of families of polytopes, in which each family gives a negative answer to Question \ref{ques:positive}.
On the other hand, constructions given in \S \ref{subsec:possible} were motivated by a refinement of Question \ref{ques:positive}, considering all possible sign patterns of the middle Ehrhart coefficients. Partial progress on the study of this question is included in \S \ref{subsec:possible}. Finally, 
}

In this section, we will discuss examples and constructions of polytopes with negative Ehrhart coefficients. 
We start in \S \ref{subsec:reeve} with the well-known Reeve tetrahedra, a family of $3$-dimensional polytopes with negative linear coefficients. 
Constructions given in \S \ref{subsec:possible} were motivated by a refinement of Question \ref{ques:positive}, considering all possible sign patterns of Ehrhart coefficients. 
Examples studied in \S \ref{subsec:smooth} and \S \ref{subsec:stanleyex} provide negative answers to Question \ref{ques:positive} for different families of polytopes (such as smooth polytopes and order polytopes), which will be summarized in \S \ref{subsec:nonpositive}. Finally in \S \ref{subsec:minkowski}, we give negative examples addressing the question of whether Minkowski summation preserves Ehrhart positivity.

As mentioned before, due to the fact that the first, second, and last Ehrhart coefficients are always positive, given a $d$-dimensional polytope $P,$ we need to ask the positivity question only for the coefficients of $t^{d-2}, t^{d-3},\dots, t^1$ in $i(P,t)$. We call these coefficients the \emph{middle Ehrhart coefficients} of $P.$

\subsection{Reeve tetrahedra} \label{subsec:reeve}
	For $d \le 2,$ there are no middle Ehrhart coefficients. Hence, possible examples with negative Ehrhart coefficients can appear only in dimension $3$ or higher. The first example comes in dimension $3:$ The \emph{Reeve tetrahedron} $T_m$ is the polytope with vertices $(0,0,0), (1,0,0), (0,1,0)$ and $(1,1,m)$, where $m$ is a positive integer. Its Ehrhart polynomial is
	\[ i(T_m,t)= \frac{m}{6} t^3+ t^2 + \frac{12-m}{6} t +1.\]
One sees that the linear coefficient is $0$ when $m=12$ and is \emph{negative} when ${m \ge 13}.$

\subsection{Possible sign patterns} \label{subsec:possible}
Motivated by the example of Reeve tetrahedra, Hibi, Higashitani, Tsuchiya and Yoshida study possible sign patterns of middle Ehrhart coefficients, and ask the following question:
\begin{ques}[Question 3.1 of \cite{HibHigTsuYos}] \label{ques:pattern}
	Given a positive integer $d \ge 3$ and integers $1 \le i_1 < \cdots < i_q \le d-2$, does there exist a $d$-dimensional integral polytope $P$ such that the coefficients of $t^{i_1}, \dots, t^{i_q}$ of $i(P,t)$ are negative, and the remaining coefficients are positive?
\end{ques}

The following is the main result in \cite{HibHigTsuYos} providing a partial answer to Question \ref{ques:pattern}.
\begin{thm}[Hibi-Higashitani-Tsuchiya-Yoshida] \label{thm:pattern}
Let $d \ge 3.$ The following statements are true.
\begin{enumerate}[(a)]
	\item \label{item:all} There exists an integral polytope $P$ of dimension $d$ such that all of its middle Ehrhart coefficients are negative.
	\item For each $1 \le k \le d-2,$ there exists an integral polytope $P$ of dimension $d$ such that $[t^k]i(P,t)$ is negative and all the remaining Ehrhart coefficients are positive.
\end{enumerate}
\end{thm}

The proof of both parts of the theorem is by construction. We will briefly discuss the construction for Theorem \ref{thm:pattern}/\eqref{item:all}, and refer interested readers to the original paper \cite{HibHigTsuYos} for the other construction.

\begin{proof}[Sketch of proof for Theorem \ref{thm:pattern}/\eqref{item:all}]
	Let $L_n :=[0,n],$ which is a $1$-dimensional polytope and its Ehrhart polynomial is $i(L_n,t) = nt +1.$ Define the polytope $P_m^{(d)}$ be the direct product of $(d-3)$ copies of $L_{d-3}$ and one copy of the Reeve tetrahedron $T_m.$ Then $P_m^{(d)}$ is a $d$-dimensional polytope with Ehrhart polynomial
	\[ i\left( P_m^{(d)}, t \right) = i(L_{d-3},t)^{d-3} \cdot i(T_m,t) = \left( (d-3)t +1 \right)^{d-3} \cdot \left( \frac{m}{6} t^3+ t^2 + \frac{12-m}{6} t +1 \right).\]
	The coefficients of the Ehrhart polynomial of $P_m^{(d)}$ can be explicitly described, from which one can show that all middle Ehrhart coefficients are negative for sufficiently large $m$.
\end{proof}

In addition to Theorem \ref{thm:pattern}, Hibi et al also show that answer to Question \ref{ques:pattern} is affirmative for $d \le 6$ \cite[Proposition 3.2]{HibHigTsuYos}.
Note that for $d \le 6,$ there are at most $3$ middle Ehrhart coefficients. 
Later, Tsuchiya (private communication) improved their result showing that any sign pattern with at most $3$ negatives is possible for the middle Ehrhart coefficients. Unfortunately, it is not currently clear how to extend the techniques used to prove this result to attack the question of whether any sign pattern with $4$ negatives can occur. So Question \ref{ques:pattern} is still wide open. 

\subsection{Smooth polytopes} \label{subsec:smooth}
A $d$-dimensional integral polytope $P$ is called \emph{smooth} (or \emph{Delzant}) if each vertex is contained in precisely $d$ edges, and the primitive edge directions form a lattice basis of $\Z^d$.
In \cite[Question 7.1]{bruns}, Bruns asked whether all smooth integral polytopes are Ehrhart positive. In \cite{alpha-cube}, Castillo, Nill, Paffenholz, and the author show the answer is false by presenting counterexamples in dimensions $3$ and higher. The main ideas we used was chiseling cubes and searching for negative BV-$\alpha$-values. 

\begin{figure}[h]
\begin{center}
\begin{tikzpicture}[scale=.7]

\draw (9,0)--(8,1)--(8,2)--(9,3)--(10,3)--(11,2)--(11,1)--(10,0)--cycle;
\draw (0,0)--(3,0)--(3,3)--(0,3)--cycle;

\node at (5.5,1.5){$\xrightarrow[\text{at distance $1$}]{\text{chisel $4$ vertices}}$};

\end{tikzpicture}
\end{center}
\caption{From $3\square_2$ to $Q_2(3,1)$.}
\label{fig:Qie}
\end{figure}
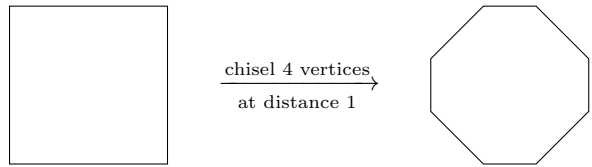

The first set of examples we construct is as follows: For positive integers $a > 2b,$ we let $Q_d(a,b)$ be the polytope obtained by chiseling \emph{all} vertices of $a \square_d$ at distance $b.$ (See Figure \ref{fig:Qie} for an example.) Using inclusion-exclusion and the fact that the BV-$\alpha$-values of cubes and standard simplices can be obtained easily due to property (P3), we obtain explicit formulas for all BV-$\alpha$-values arising from $P_d(a,b),$ which we use to search for negative BV-$\alpha$-values. The first negative values appear at $d=7$, suggesting that we might have a negative Ehrhart coefficient in $Q_7(a,b)$. By direct computation, we are able to show that for some choices of $(a,b)$, e.g., $(5,2),$ the polytope $Q_d(a,b)$ has a negative linear Ehrhart coefficient for any $d \ge 7.$ Therefore, we have the following result \cite[Proposition 1.3]{alpha-cube}:
\begin{prop}[Castillo-L.-Nill-Paffenholz]\label{prop:Q}
	Let $\cN_d$ be the normal fan of $Q_d(a,b).$ 
For $d \le 6$, any $d$-dimensional smooth integral polytope with normal fan $\cN_d$ is Ehrhart positive. For any $d \ge 7$, there exists a $d$-dimensional smooth integral polytope with normal fan $\cN_d$ whose linear Ehrhart coefficient is negative.
\end{prop}

\begin{rem} \label{rem:typeB}
	The polytope $Q_d(a,b)$ is not only a smooth polytope, but also a ``type-$B$ generalized permutohedron''. The generalized permutohedra considered in Section \ref{sec:mcmullen} are of type $A$ as the corresponding normal fan, $Br_d$, is constructed from the type $A$ root system. As a consequence, a polytope $P$ is a (type-$A$) generalized permutohedron if and only if each edge direction of $P$ is of the form of $\be_i -\be_j$ for some $i \neq j$. 
	Similarly, we can define a polytope $P$ in $\R^{d}$ is a \emph{type-$B$ generalized permutohedron} if each edge direction of $P$ is in the form of $\be_i\pm\be_j$ for some $i \neq j$ or of the form $\pm \be_i$ for some $i.$ It is then straightforward to verify that $Q_d(a,b)$ is a type-$B$ generalized permutohedron. 
\end{rem}

Using the idea of iterated chiseling cubes, we then improve the dimension range of our counterexamples from $d \ge 7$ to $d \ge 3.$ (See \cite[Section 2]{alpha-cube} for details.)
\begin{thm}[Castillo-L.-Nill-Paffenholz] \label{thm:smooth}
For any $d \ge 3$, there exists a $d$-dimensional smooth integral polytope $P$ such that all of its middle Ehrhart coefficients are negative.  
\end{thm}
Note that the above theorem is a stronger version than part \eqref{item:all} of Theorem \ref{thm:pattern}.
Even though the original purpose of the paper \cite{alpha-cube} was to answer Bruns' question, 
in the process of searching for a counterexample, we obtained a separate result answering a different question. For positive integers $a >b,$ we let $P_d(a,b)$ be the polytope obtained by chiseling \emph{one} vertex off $a \square_d$ at distance $b.$ It is clear that $P_d(a,b)$ share the same BV-$\alpha$-values with $Q_d(a,b).$ Hence, it has negative BV-$\alpha$-values at $d \ge 7.$ However, it turns out any $d$-dimensional integral polytope $P$ that has the same normal fan as $P_d(a,b)$ is Ehrhart positive \cite[Lemma 3.9]{alpha-cube}. 
\begin{cor}[Castillo-L.-Nill-Paffenholz]
	For $d \ge 7,$ there exists a smooth projective fan $\Sigma$, such that its associated BV-$\alpha$-values contains negative values, but any smooth integral polytope in $\Poly(\Sigma)$ is Ehrhart positive. 

	Therefore, BV-$\alpha$-positivity is strictly stronger than Ehrhart-positivity.
\end{cor}

Finally, 
we studied a weaker version of Brun's question by requiring the smooth polytopes to be reflexive. More precisely, we asked whether all smooth reflexive polytopes have positive Ehrhart coefficients. Unfortunately, the answer to this question is still negative.

In fixed dimension $d,$ there are only finitely many reflexive polytopes up to unimodular transformations. Because of their correspondence to toric Fano manifolds, smooth reflexive polytopes were completely classified up to dimension $9$ \cite{Oebro, LorPaf-Database}. 
We used \texttt{polymake} \cite{Polymake} to verify that up to dimension $8$ all of them are Ehrhart positive, but in dimension $9$ the following counterexample came up \cite{alpha-cube}: 
\begin{ex}\label{ex:smoothreflexive}
Let $P$ be the polytope in $\R^9$ defined by
\[ \begin{pmatrix}
\scriptstyle
1 & 0 & 0 & 0 & 0 & 0 & 0 & 0 & 0 \\
0 & 1 & 0 & 0 & 0 & 0 & 0 & 0 & 0 \\
0 & 0 & 1 & 0 & 0 & 0 & 0 & 0 & 0 \\
0 & 0 & 0 & 1 & 0 & 0 & 0 & 0 & 0 \\
0 & 0 & 0 & 0 & 1 & 0 & 0 & 0 & 0 \\
0 & 0 & 0 & 0 & 0 & 1 & 0 & 0 & 0 \\
0 & 0 & 0 & 0 & 0 & 0 & 1 & 0 & 0 \\
0 & 0 & 0 & 0 & 0 & 0 & 0 & 1 & 0 \\
0 & 0 & 0 & 0 & 0 & 0 & 0 & 0 & 1 \\
0 & 0 & 0 & 0 & 0 & 0 & 0 & 0 & -1 \\
-1 & -1 & -1 & -1 & 0 & 0 & 0 & 0 & 4 \\
0 & 0 & 0 & 0 & -1 & -1 & -1 & -1 & -4 
	\end{pmatrix} 
	\begin{pmatrix} x_1 \\ x_2 \\ x_3 \\ x_4 \\ x_5 \\ x_6 \\ x_7 \\ x_8 \\ x_9 \end{pmatrix}
	\le 
	\begin{pmatrix} 1 \\ 1 \\ 1 \\ 1 \\ 1 \\ 1 \\ 1 \\ 1 \\ 1 \end{pmatrix}
\]
Using \texttt{polymake} \cite{Polymake}, one can check that this polytope is smooth and reflexive with Ehrhart polynomial 
\begin{align*}
i(P,t) =& \ 12477727/18144 t^9 + 12477727/4032 t^8 + 9074291/1512 t^7 +  630095/96 t^6 \\
& \ + 19058687/4320 t^5 +  117857/64 t^4 + 3838711/9072 t^3 + 11915/1008 t^2 \\
& \ -6673/630  t + 1,
\end{align*}
which has a negative linear coefficient.
\end{ex}

\subsection{Stanley's example} \label{subsec:stanleyex}
In answering an Ehrhart positivity question posted on mathoverflow, Stanley gave the following example \cite{stanley-mathoverflow2015}:
\begin{ex}\label{ex:stanleyex}
	Let $\cQ_k$ be the poset with one minimal element covered by $k$ other elements. The Ehrhart polynomial of the order polytope $\cO(\cQ_k)$ is 
	\[ i(\cO(\cQ_k), t) = \sum_{i=1}^{t+1} i^k.\]
One can compute that
\[ [t^1] i(\cO(\cQ_{20}), t) = -168011/330 < 0.\]
Hence, the linear Ehrhart coefficient of $\cO(\cQ_k)$ is negative when $k=20.$ 
\end{ex}

Based on Stanley's example, the author and Tsuchiya studied the Ehrhart positivity question on polytopes $\cO(\cQ_k)$ for any $k$, and gave a complete description of which Ehrhart coefficients of $\cO(\cQ_k)$ are negative \cite{stanley-order}. The following theorem is an immediate consequence to this description.
\begin{thm}[L.-Tsuchiya]
	The order polytope $\cO(\cQ_k)$ (defined in Example \ref{ex:stanleyex}) is Ehrhart positive if and only if $k \le 19.$
\end{thm}

Stanley's example and its extension are very interesting as $\cO(\cQ_{k})$ belongs to a lot of different families of polytopes. First of all, it is an order polytope, and thus is a $(0,1)$-polytope. 

Recall that a Gorenstein polytope of codegree $s$ is an integral polytope such that $sP$ is reflexive. 
It follows from a result by Hibi \cite{hibi1987} that an order polytope is Gorenstein if and only if the underlying poset is \emph{pure}, i.e., all maximal chains have the same length. Clearly, $\cQ_{k}$ is pure. Thus, $\cO(\cQ_{k})$ is a Gorenstein polytope. 

Finally, 
M\'esz\'aros, Morales and Striker proved a result observed by Postnikov establishing a connection between flow polytopes of planar graphs and order polytopes \cite[Theorem 3.8]{MesMorStr}. 
Using this connection, Morales (private communication) observes that the order polytope $\cO(\cQ_{k})$ is unimodularly equivalent to the flow polytope $\cF_G(1, 0, \dots, 0, -1)$, 
where $G$ is the black graph on $k+1$ vertices in Figure \ref{fig:flow-order}. (Note the red part of the figure is $\cQ_{k}$.)
\begin{figure}[htp]
	\includegraphics[scale=0.7, trim = 0 680 100 50, clip]{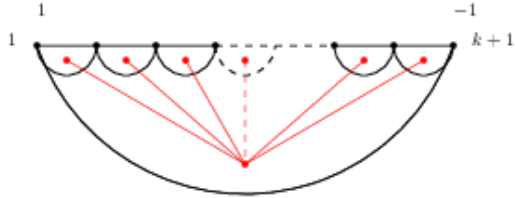}
	\caption{$\cQ_{k}$ and its corresponding planar graph $G$.}
 \label{fig:flow-order}
\end{figure}

\subsection{Non-Ehrhart-positive families}\label{subsec:nonpositive}
	For each of the families listed below, it is \emph{not} true that all the integral polytopes in the family are Ehrhart positive.
	\begin{multicols}{2}
		\begin{enumerate}[(i)]
			\item \label{item:smooth} Smooth polytopes. 
			\item \label{item:typeB} Type-$B$ generalized permutohedra.
			\item \label{item:01} $(0,1)$-polytopes. 
			\item \label{item:order} Order polytopes.
			\item \label{item:chain} Chain polytopes. 
			\item \label{item:flow} Flow polytopes.
			\item \label{item:gorenstein} Gorenstein polytopes. 
			\item \label{item:reflexive} Reflexive polytopes.
			\item \label{item:smoothreflexive} Smooth reflexive polytopes.
		\end{enumerate}
	\end{multicols}
	Furthermore, non-Ehrhart-positive examples were constructed for family \eqref{item:smooth} for each dimension $d \ge 3$, for family \eqref{item:typeB} for each dimension $d \ge 7,$ and for families \eqref{item:01}, \eqref{item:order}, \eqref{item:chain}, \eqref{item:flow}, \eqref{item:gorenstein} and \eqref{item:reflexive} for each dimension $d \ge 21.$ 
	\begin{proof}
		The conclusion for \eqref{item:smooth} and \eqref{item:typeB} follows from Theorem \ref{thm:smooth}, Proposition \ref{prop:Q} and Remark \ref{rem:typeB}.
		Notice that the order polytope $\cO(\cQ_k)$ considered in \S \ref{subsec:stanleyex} has dimension $k+1.$ Then the conclusion for \eqref{item:01}, \eqref{item:order}, \eqref{item:flow}, \eqref{item:gorenstein} follows directly from discussion in \S \ref{subsec:stanleyex}. 
		Next, \eqref{item:chain} follows from \eqref{item:order} and Remark \ref{rem:chain},	
		and \eqref{item:reflexive} follows from \eqref{item:gorenstein}, the connection between Gorenstein polytopes and reflexive polytopes and the fact that Ehrhart positivity is invariant under dilating operations. 
		Finally, \eqref{item:smoothreflexive} follows from Example \ref{ex:smoothreflexive}.
	\end{proof}

	\subsection{Minkowski sums}  \label{subsec:minkowski}
Recall that in \S \ref{subsubsec:typey}, we learned that the type-$\cY$ generalized permutohedra which are defined to be Minkowski sums of dilated standard simplices are Ehrhart positive. Noticing that standard simplices are Ehrhart positive (see \S \ref{simplex}), we asked the following question in the first version of this survey:
\begin{center}
	\emph{Is it true that if two integral polytopes $P$ and $Q$ are Ehrhart positive, \\ then their Minkowski sum $P + Q$ is Ehrhart positive?}
\end{center}
Tsuchiya (private communication) constructed a few examples, which gave a negative answer to the above question, shortly after it was posted. Below are two of his examples. 
\begin{ex}\label{ex:mink1}
Let $P$ be the $3$-dimensional simplex with vertices 
\[(0,0,0,0), (1,0,0,0), (0,1,0,0), (0,0,1,0),\] 
and $Q$ the $1$-dimensional polytope with vertices 
\[ (0,0,0,0), (1, 19, 19, 20).\]
It is easy to see that both $P$ and $Q$ are Ehrhart positive.
However, one can check that $P+Q$ is a $4$-dimensional polytope 
with Ehrhart polynomial
\[ i(P+Q,t) = 10/3t^4 + 7/6 t^3 - 1/3 t^2 + 17/6 t +1,\]
which has a negative quadratic coefficient.
\end{ex}
\begin{ex}\label{ex:mink2}
Let $P$ be the $5$-dimensional simplex with vertices 

\[(0,0,0,0,0), (1,0,0,0,0), (0,1,5,15,16), (0,0,1,0,0), (0,0,0,1,0), (0,0,0,0,1),\] 
and $Q$ the $5$-dimensional simplex with vertices 
\[(0,0,0,0,0), (1,0,0,0,0), (1,1,15,15,16), (0,0,1,0,0), (0,1,0,0,0), (0,0,0,0,1).\]
Since $P$ is unimodularly equivalent to the standard $5$-simplex, it is Ehrhart positive. Moreover, the Ehrhart polynomial of $Q$ is
\[ i(Q,t) = 1/8 t^5 + 5/12 t^4 + 17/24 t^3 + 19/12 t^2 + 13/6 t + 1,\]
which also has positive coefficients.

However, $P+Q$ is a $5$-dimensional polytope with Ehrhart polynomial  
\[ i(P+Q,t) = 3007/40 t^5 + 359/24 t^4 - 255/24 t^3 + 193/24 t^2 + 89/20 t + 1,\]
which has a negative coefficient.
\end{ex}

Notice that the polytopes given in Example \ref{ex:mink1} satisfy $\dim(P) + \dim(Q) = \dim(P+Q),$ and those in Example \ref{ex:mink2} satisfy $\dim(P)=\dim(Q)=\dim(P+Q).$
These are the two extreme situations in terms of dimensions. 
Therefore, even with some restrictions on the dimensions of $P$, $Q$ and $P+Q,$ the answer to the question above is false.

\commentout{
	In particular, the $d$-dimensional standard simplex is included. Therefore, we pose the following revised question: 
\begin{ques}
Is it true that if two integral polytopes $P$ and $Q$ are Ehrhart positive, then their Minkowski sum $P + Q$ is Ehrhart positive as long as either $\dim(P)$ or $\dim(Q)$ is equal to $\dim(P+Q)$?
\end{ques}

The answer to the above question, if true, would provide another explanation for the Ehrhart positivity of type-$\cY$ generalized permutohedra, as well as give a new way of attacking Ehrhart positivity problems. 
}
	
\section{Further discussion} \label{sec:future}

\subsection{Ehrhart positivity conjectures}

We list several families of polytopes that are conjectured to be Ehrhart positive.

\subsubsection{Base-$r$ simplices} Recall the definition of $\Delta_{(1, \bq)}$ given in \S \ref{subsubsec:delta1q}. For any positive integer $r \in \P,$ we let $\bq_r := (r-1, (r-1)r, (r-1)r^2, \dots, (r-1)r^{d-1}) \in \P^d,$ and then define the \emph{base-$r$ $d$-simplex} to be
\[ \cB_{(r,d)} := \Delta_{(1,\bq_r)}.\]
Note that if $r=1,$ we obtain a polytope that is unimodularly equivalent to the standard $d$-simplex. The family of base-$r$ $d$-simplices are introduced by Solus in his study of simplices for numeral systems \cite{solus}, in which he shows that the $h^*$-polynomial of $\cB_{(r,d)}$ is real-rooted and thus is unimodal. Based on computational evidence, Solus makes the following conjecture \cite[Section 5]{solus}:
\begin{conj}[Solus]
The base-$r$ $d$-simplex is Ehrhart positive. 
\end{conj}

We remark that this family of $\Delta_{(1,\bq)}$ is very different from the ones constructed by Payne discussed in Example \ref{ex:payne}. 
If $r > 1,$ the base-$r$ simplex $\cB_{(r,d)}$ always contains the origin as an interior point, and it follows from \eqref{equ:knownh} that the degree of $h^*$-polynomial of $\cB_{(r,d)}$ is $d.$
Since $\cB_{(r,d)}$ is not reflexive, by Corollary \ref{cor:gorenstein} the roots of its $h^*$-polynomial are not all on the unit circle in the complex plane. 
Therefore, the techniques used to prove Ehrhart positivity for Payne's construction would not work here. 

\subsubsection{Birkhoff polytopes}
The {\it Birkhoff polytope} $B_n$ is the convex polytope of $n \times n$ doubly-stochastic matrices; that is, the set of real nonnegative matrices with all row and column sums equal to one. Equivalently, $B_n$ can also be defined as the convex hull of all $n \times n$ permutation matrices. (See \cite[Chapters 5 and 6]{yemelichevKK} for a detailed introduction to $B_n.$)
There has been a lot of research on computing the volumes and Ehrhart polynomials of Birkhoff polytopes \cite{beckpixton, canfield-mckay2007b, birkhoff, pak2000}. 
The following conjecture was made by Stanley in a talk \cite{stanleyehrhart1}:
\begin{conj}[Stanley]
Birkhoff polytopes are Ehrhart positive.
\end{conj}

By checking the available data \cite{beckpixton}, the first nine $i(B_n,t)$ have the property that all the roots have negative real parts. More importantly, Figure 6 in \cite{beck-deloera-et2005} suggests that the roots of $i(B_n,t)$ form a certain pattern. 
Hence, it could be promising to use Lemma \ref{lem:neg} to attack this conjecture. 

We also remark that $B_n$ is a Gorenstein polytope (up to lattice translation) of codegree $n$. However, with aforementioned data, one can see that $B_n$ is not $h^*$-unit-circle-rooted. Hence, we cannot apply Theorem \ref{thm:root} to show that all roots of $i(B_n,t)$ have negative real parts.   

\subsubsection{Tesler polytopes}
For any $n \times n$ upper triangular matrix $M=(m_{i,j})$, the $k^{th}$ \emph{hook sum} of $M$ is the sum of all the elements on the $k^{th}$ row minus the sum of all the elements on the $k^{th}$ column excluding the diagonal entry:
\[ (m_{k,k} + m_{k,k+1} + \cdots + m_{k,n}) - (m_{1,k} + m_{2,k} + \cdots + m_{k-1,k}).\]
For each $\ba =(a_1, \dots, a_n) \in \N^n$, M\'esz\'aros, Morales, and Rhoades \cite{Tesler} define the \emph{Tesler polytope}, denoted by $\tes_n(\ba),$ to be the set of all $n\times n$ upper triangular matrices $M$ with nonnegative entries and of \emph{hook sum} $\ba$, i.e., the $k^{th}$ hook sum of $M$ is $a_k.$
The lattice points in $\tes_n(\ba)$ are exactly \emph{Tesler matrices of hook sum $\ba$}. When $\ba=(1, 1,\dots,1),$ these are important objects in Haglund's work on diagonal harmonics \cite{haglund}. Therefore, we call $\tes_n(1, 1, \dots, 1)$ the \emph{Tesler matrix polytope} as Tesler matrices of hook sum $(1, 1, \dots, 1)$ were the original Tesler matrices defined by Haglund.

Another interesting example of a Tesler polytope is $\tes_n(1, 0, \dots, 0)$, which turns out to be the \emph{Chan-Robbins-Yuen} polytope or \emph{CRY} polytope, a face of the Birkhoff polytope. The \emph{CRY} polytope, denote by $\CRY_n,$ is the convex hull of all the $n \times n$ permutation matrices $M=(m_{i,j})$ such that $m_{i,j} = 0$ if $i \ge j+2,$ i.e., all entries below the sub-diagonal are zeros.
It was initially introduced by Chan, Robbins and Yuen in \cite{chanrobbinsyuen}, in which they made an intriguing conjecture on a formula for the volume of $\CRY_n$ as a product of Catalan numbers. It was since proved by Zeilberger \cite{zeilberger}, Baldoni-Vergne \cite{BalVer}, and M\'{e}sz\'aros \cite{meszaros, meszaros-prod}.

Using the Maple code provided by Baldoni, Beck, Cochet and Vergne \cite{BalBecCocVerMaple}, Morales computed the Ehrhart polynomials of both CRY polytopes and Tesler matrix polytopes for small $n$, and made the following conjecture \cite{flowehrhartweb}.

\begin{conj}[Morales]\label{conj:flow}
For each $n,$ the CRY polytope $\CRY_n = \tes_n(1, 0, \dots, 0)$ and the Tesler matrix polytope $\tes_n(1, 1, \dots, 1)$ are both Ehrhart positive.
\end{conj}

\subsection*{Connection to flow polytopes}
M\'esz\'aros et al show in \cite[Lemma 1.2]{Tesler} that for any $\ba \in \N^n,$ the Tesler polytope $\tes_n(\ba)$ is unimodularly equivalent to the flow polytope $\cF_{K_{n+1}}(\bar{\ba})$, 
where $K_{n+1}$ is the complete graph on $[n+1]$ and $\bar{\ba}$ is defined as in \eqref{equ:bara}. Therefore, Tesler polytopes are flow polytopes associated to complete graphs. Note that the complete graph does not satisfy the hypothesis of Corollary \ref{cor:flow}. So Conjecture \ref{conj:flow} does not follow.

\subsubsection{Stretched Littlewood-Richardson coefficients}
The Schur functions $s_\lambda$ form a basis for the ring of symmetric functions. (See \cite[Chaprter 7]{stanleyec2} for background on symmetric functions.) Therefore, the product of two Schur functions $s_\lambda$ and $\s_\mu$ can be uniquely expressed as
\[ s_\lambda \cdot s_\mu = \sum_{\nu: |\nu| = |\lambda| + |\mu|} c_{\lambda, \mu}^{\nu} s_\nu.\]
We call the coefficients $c_{\lambda,\mu}^{\nu}$ in the above expression the \emph{Littlewood-Richardson coefficients} or \emph{LR coefficients}. 
There are many different ways of computing $c_{\lambda,\mu}^{\nu}.$ For example, it counts the number of semistandard Young tableaux $T$ of shape $\nu/\lambda$ with content $\mu$ such that the reading word of $T$ satisfies the ``Yamanouchi word condition'' \cite{macdonald-symmetric}.
One immediate consequence of these descriptions is that the LR coefficients are nonnegative integers. In this article, we use the hive model \cite{buch, KnuTao, KnuTaoWoo} to describe the LR coefficients.

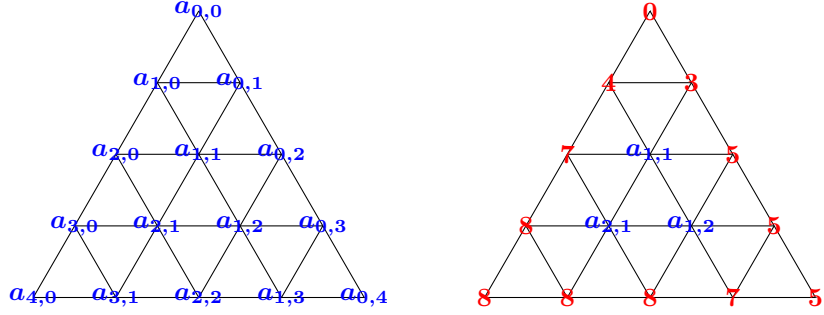
\begin{figure}
\newcommand*\rows{4}
\begin{tikzpicture}
	\begin{scope}[scale=1.1]
    \foreach \row in {0, 1, ...,\rows} {
	    \draw[thin] ($\row*(0.5, {0.5*sqrt(3)})$) -- ($(\rows,0)+\row*(-0.5, {0.5*sqrt(3)})$);
	    \draw[thin] ($\row*(1, 0)$) -- ($(\rows/2,{\rows/2*sqrt(3)})+\row*(0.5,{-0.5*sqrt(3)})$);
	    \draw[thin] ($\row*(1, 0)$) -- ($(0,0)+\row*(0.5,{0.5*sqrt(3)})$);
    }

    \node at (0, 0) {$\textcolor{blue}{\bm{a_{4,0}}}$};
    \node at (1, 0) {$\textcolor{blue}{\bm{a_{3,1}}}$};
    \node at (2, 0) {$\textcolor{blue}{\bm{a_{2,2}}}$};
    \node at (3, 0) {$\textcolor{blue}{\bm{a_{1,3}}}$};
    \node at (4, 0) {$\textcolor{blue}{\bm{a_{0,4}}}$};
    \node at (0.5, 0.866) {$\textcolor{blue}{\bm{a_{3,0}}}$};
    \node at (1.5, 0.866) {$\textcolor{blue}{\bm{a_{2,1}}}$};
    \node at (2.5, 0.866) {$\textcolor{blue}{\bm{a_{1,2}}}$};
    \node at (3.5, 0.866) {$\textcolor{blue}{\bm{a_{0,3}}}$};
    \node at (1, 1.732) {$\textcolor{blue}{\bm{a_{2,0}}}$};
    \node at (2, 1.732) {$\textcolor{blue}{\bm{a_{1,1}}}$};
    \node at (3, 1.732) {$\textcolor{blue}{\bm{a_{0,2}}}$};
    \node at (1.5, 2.598) {$\textcolor{blue}{\bm{a_{1,0}}}$};
    \node at (2.5, 2.598) {$\textcolor{blue}{\bm{a_{0,1}}}$};
    \node at (2, 3.464) {$\textcolor{blue}{\bm{a_{0,0}}}$};
    \end{scope}

    	\begin{scope}[xshift=6cm, scale=1.1]
    \foreach \row in {0, 1, ...,\rows} {
	    \draw[thin] ($\row*(0.5, {0.5*sqrt(3)})$) -- ($(\rows,0)+\row*(-0.5, {0.5*sqrt(3)})$);
	    \draw[thin] ($\row*(1, 0)$) -- ($(\rows/2,{\rows/2*sqrt(3)})+\row*(0.5,{-0.5*sqrt(3)})$);
	    \draw[thin] ($\row*(1, 0)$) -- ($(0,0)+\row*(0.5,{0.5*sqrt(3)})$);
    }

    \node at (0, 0) {$\textcolor{red}{\bm{8}}$};
    \node at (1, 0) {$\textcolor{red}{\bm{8}}$};
    \node at (2, 0) {$\textcolor{red}{\bm{8}}$};
    \node at (3, 0) {$\textcolor{red}{\bm{7}}$};
    \node at (4, 0) {$\textcolor{red}{\bm{5}}$};
    \node at (0.5, 0.866) {$\textcolor{red}{\bm{8}}$};
    \node at (1.5, 0.866) {$\textcolor{blue}{\bm{a_{2,1}}}$};
    \node at (2.5, 0.866) {$\textcolor{blue}{\bm{a_{1,2}}}$};
    \node at (3.5, 0.866) {$\textcolor{red}{\bm{5}}$};
    \node at (1, 1.732) {$\textcolor{red}{\bm{7}}$};
    \node at (2, 1.732) {$\textcolor{blue}{\bm{a_{1,1}}}$};
    \node at (3, 1.732) {$\textcolor{red}{\bm{5}}$};
    \node at (1.5, 2.598) {$\textcolor{red}{\bm{4}}$};
    \node at (2.5, 2.598) {$\textcolor{red}{\bm{3}}$};
    \node at (2, 3.464) {$\textcolor{red}{\bm{0}}$};
    \end{scope}
\end{tikzpicture}
\caption{A hive of size $4$.}
\label{fig:hive}
\end{figure}
A \emph{hive of size $n$} 
is a triangular array of numbers $a_{i,j}$ with $0 \le i, j, i+j \le n$ arranged on a triangular grid consisting of $n^2$ small equilateral triangles. See the left side of Figure \ref{fig:hive} for how a hive of size $4$ should look like.
For any adjacent triangles $\{a,b,c\}$ and $\{b,c,d\}$ in the hive, they form a rhombus $\{a,b,c,d\}.$ The \emph{hive condition} for this rhombus is 
\begin{equation}\tag{HC}
	b + c \ge a + d. \label{equ:hc}
\end{equation}
Suppose $|\nu| = |\lambda| + |\mu|$ with $l(\nu), l(\lambda), l(\mu) \le n.$ A \emph{Littlewood-Richardson-hive} or \emph{LR-hive} of \emph{type $(\nu, \lambda, \mu)$} is a hive $\{a_{i,j} \in \N \ : \ 0 \le i,j, i+j\le n\}$ with nonnegative integer entries satisfying the hive condition \eqref{equ:hc} for all of its rhombi, with border entries determined by partitions $\nu, \lambda, \mu$ in the following way: 
$a_{0,0}=0$ and for each $j=1,2,\dots,n$,
\[ a_{j,0}-a_{j-1,0} = \nu_j, \quad a_{0,j}-a_{0,j-1} = \lambda_j, \quad  a_{j,n-j} - a_{j-1,n-j+1} = \mu_k.\]
With this definition, the LR-coefficient $c_{\lambda, \mu}^\nu$ counts the number of LR-hives of type $(\nu, \lambda, \mu).$ (Note that this is independent from $n$ as long as $l(\nu), l(\lambda), l(\mu) \le n.$)

For example, if $\nu=(4,3,1)$, $\lambda=(3,2)$ and $\mu=(2,1),$ then the border of a corresponding LR-hive of size $4$ is shown on the right side of Figure \ref{fig:hive}. In fact, the hive condition will force $a_{2,1}=8$ and $a_{1,2}=7$. So it will be reduced to a hive of size $3.$ Finally, it follows from the hive condition that $6 \le a_{1,1} \le 7.$ Thus, we have two LR-hives of this type, and we conclude that $c_{(3,2),(2,1)}^{(4,3,1)}=2.$ 

From the above description, it is not hard to see that $c_{\lambda, \mu}^{\nu}$ counts the number of lattice points inside a polytope $P_{\lambda, \mu}^{\nu}$ determined by the border condition and the hive condition. Furthermore, for any positive integer $t,$ the LR-coefficient $c_{t\lambda, t\mu}^{t\nu}$ counts the number of lattice points inside the $t^{th}$ dilation of $P_{\lambda, \mu}^{\nu}:$
\[ c_{t \lambda, t\mu}^{t\nu} = | t P_{\lambda, \mu}^{\nu} \cap \Z^D|.\]

King, Tollu and Toumazet studied $c_{t \lambda, t\mu}^{t\nu}$, which they call the \emph{stretched Littlewood-Richardson coefficients}, and made the following conjecture \cite[Conjecture 3.1]{KinTolTou}:
\begin{conj}[King-Tollu-Toumazet]\label{conj:stretch}
	For all partitions $\lambda, \mu, \nu$ such that $c_{\lambda, \mu}^{\nu} >0$, there exists a polynomial $f(t) = f_{\lambda,\mu}^{\nu}(t)$ in $t$ such that $f(0)=1$ and $f(t) = c_{t \lambda, t\mu}^{t\nu}$ for all positive integers $t.$

	Furthermore, all the coefficients of $f(t)$ are positive.
\end{conj}
One notices that if $P_{\lambda, \mu}^{\nu}$ is an integral polytope, then the polynomiality part of the above conjecture follows from Ehrhart's theorem. 
However, in general, $P_{\lambda, \mu}^{\nu}$ is a rational polytope, which only implies that $c_{t \lambda, t\mu}^{t\nu}$ is a quasi-polynomial with some period.
Nevertheless, the assertion of polynomiality in the above conjecture was established first by Derksen and Weyman \cite{DerWey} using semi-invariants of quivers, and then by Rassart \cite{rassart} using Steinberg's formula \cite{steinberg} and hyperplane arrangements. Hence, the polynomial asserted in Conjecture \ref{conj:stretch} can be considered to be an Ehrhart polynomial, and positivity assertion in the conjecture (which is still open) is exactly an Ehrhart-positivity question. 



\subsection{Other questions}
Many questions related to Ehrhart positivity remain open. We include a few below.

\subsubsection{Modified Bruns question}
As we have discussed in \S \ref{subsec:smooth}, the answer to Bruns' question of whether all smooth polytopes are Ehrhart positive is negative, where counterexamples are constructed for each dimension $d \ge 3.$ 
Furthermore, we verify, with the help of \texttt{polymake} \cite{Polymake}, that all smooth reflexive polytopes of dimension up to $8$ are Ehrhart positive, and that there exists a non-Ehrhart-positive smooth reflexive polytopes of dimension $9$. However, we did not investigate smooth reflexive polytopes of higher dimensions. 
Therefore, one can ask:

\begin{ques} 
	Does there exist a smooth reflexive polytope of dimension $d$ with negative Ehrhart coefficients, for any $d \ge 10$?
\label{ques:smoothreflexive}
\end{ques}

Bruns' question can be rephrased using the language of fans: For any smooth projective fan $\Sigma,$ is it true that \emph{any} polytope with normal fan $\Sigma$ is Ehrhart positive. Since the answer is false, a weaker version of this question can be asked: 
\begin{ques}
	Is it true that for any smooth projective fan $\Sigma,$ there exists \emph{one} integral polytope $P$ with normal fan $\Sigma$ that is Ehrhart positive? 
\end{ques}

\subsubsection{$h^*$-vector for $3$-dimensional polytopes}
Here instead of studying Ehrhart positivity question for families of polytopes in which dimensions vary, we focus on polytopes with a fixed dimension, and ask the following question:
\begin{ques}\label{ques:likely}
For each $d,$ how likely is an integral polytope Ehrhart positive?	
\end{ques}
Since integral polytopes of dimension at most $2$ are always Ehrhart positive, dimension $3$ is a natural starting point. 

We have mentioned in the introduction that various inequalities for $h^*$-vectors have been found. So we may use Formula \eqref{equ:h2e} which gives a connection between the $h^*$-vector and Ehrhart coefficients together with known inequalities to study Question \ref{ques:likely}. Note that in dimension $3,$ only the linear Ehrhart coefficient could be negative. Applying \eqref{equ:h2e}, we obtain that $P$ is Ehrhart positive (equivalently the linear Ehrhart coefficient of $P$ is positive) if and only if the $h^*$-vector $(h_0^*, h_1^*, h_2^*, h_3^*)$ of $P$ satisfies 
\begin{equation}
	11 h_0^* + 2 h_1^* - h_2^* + 2h_3^* > 0.
	\label{equ:ineq}
\end{equation}
In \cite{BalKas}, Balletti and Kasprzyk give classifications for $3$-dimensional polytopes with $1$ or $2$ interior lattice points, using which they extract all possible $h^*$-vectors. Assume the number of interior lattice points is fixed to be $1$ or $2.$ Applying \eqref{equ:knownh}, we obtain $h_0^*=1$ and $h_3^*= 1$ or $2$. Hence, only $h_1^*$ and $h_2^*$ change. Balletti and Kasprzyk then plot all occurring pairs of $(h_1^*, h_2^*)$ in \cite[Figure 5]{BalKas}. 
The black part of Figure \ref{fig:hplot} is their figure,
which we modify to include a red line representing the inequality \eqref{equ:ineq}, where points below the red line arise from polytopes with the Ehrhart positivity property.
\begin{figure}
\centering
\hfill
\subfigure[Polytopes with one interior point]{
\begin{tikzpicture}[trim axis left, trim axis right]
\begin{axis}[xmin=0,xmax=42,ymin=0,ymax=75,xlabel=$\h_1$,ylabel=$\h_2$,width=14cm,axis lines=middle,axis equal image,grid=both]

\addplot[black,mark=none,domain=0.5:38, thick]
{x} node[below,pos=0.5,yshift=22pt]{};

\addplot[black,dashed,mark=none,domain=0:38, thick]
{-x+70} node[below,pos=0.87]{};

\addplot[red,mark=none,domain=0.5:19, thick]
{2*1+2*x+11} node{};

\draw[black, thick] (axis cs:1,0) -- (axis cs:1,71) node [below,pos=0.94,xshift=9]{};
\pgfplotstableread{3_1.txt}\mydata;
\addplot [only marks,mark=*,mark options={scale=0.8,fill=white}] table [] {\mydata};

\pgfplotstableread{3s_1.txt}\mydata;
\addplot [color=black,only marks,mark=*,mark options={scale=0.8}] table [] {\mydata};
\end{axis}
\end{tikzpicture}
}
\hfill
\subfigure[Polytopes with two interior points]{
\begin{tikzpicture}[trim axis left, trim axis right]
\begin{axis}[xmin=0,xmax=65,ymin=0,ymax=111,xlabel=$\h_1$,ylabel=$\h_2$,width=14cm,axis lines=middle,axis equal image,grid=both]

\addplot[black,mark=none,domain=0:58, thick]
{x} node[below,pos=0.5,yshift=15pt]{};



\addplot[black,dashed,mark=none,domain=0:58, thick]
{-x+105} node[below,pos=0.87]{};

\addplot[red,mark=none,domain=0.5:30, thick]
{2*2+2*x+11} node{};


\draw[black, thick] (axis cs:2,0) -- (axis cs:2,105) node [below,pos=0.94,xshift=9]{};
\pgfplotstableread{3_2.txt}\mydata;
\addplot [only marks,mark=*,mark options={scale=0.5,fill=white}] table [] {\mydata};
\pgfplotstableread{3s_2.txt}\mydata;
\addplot [color=black,only marks,mark=*,mark options={scale=0.5}] table [] {\mydata};
\end{axis}
\end{tikzpicture}
}
\hfill\phantom{.}
     \caption{The plot of $(h_1^*, h_2^*)$ of $3$-dimensional polytopes with $1$ or $2$ interior lattice points.}
 \label{fig:hplot}
\end{figure}
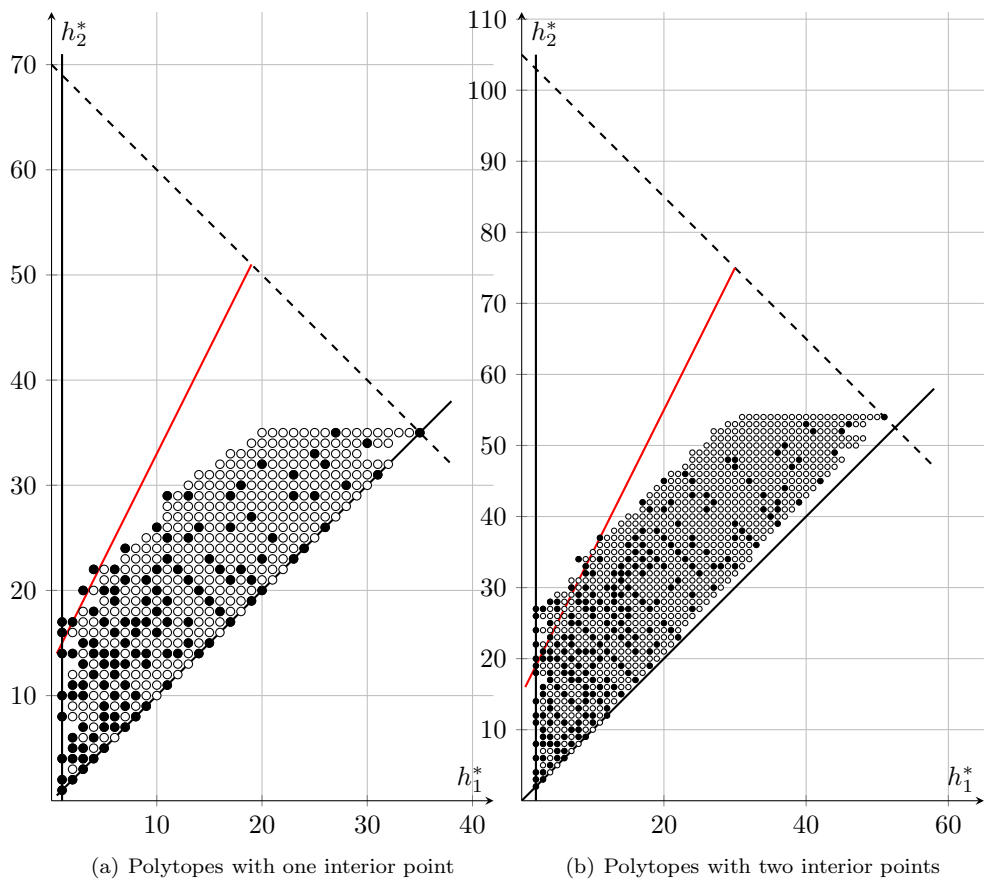
 Note that in each part of the figure, the big triangular area is bounded by three known inequalities for $h^*$-vectors. It is clear from the figure that these inequalities are far from optimal. Comparing the red line with the plotted data, one sees that a very high percentage of data points correspond to Ehrhart positive polytopes. However, if we only look at the triangular region (without the data points), then the area below the red line has a much lower percentage of the region. Therefore, improving the inequality bounds for $h^*$-vectors will be helpful in understanding the Ehrhart positivity problem, in particular, in giving a more accurate answer to Question \ref{ques:likely}.
 
 \bibliographystyle{amsplain}
\bibliography{gen}

\end{document}